\documentclass[10pt]{article}

\usepackage{tikz}
\usetikzlibrary{matrix,arrows,decorations.pathmorphing,shapes.geometric}
\usepackage{tikz-cd}

\usepackage[utf8]{inputenc}
\usepackage{a4wide}
\usepackage{graphicx}
\usepackage[hmarginratio={1:1},     
vmarginratio={1:1},     
textwidth=16cm,        
textheight=22cm,
heightrounded,]{geometry}

\usepackage{amsmath,accents}
\usepackage{amsthm}
\usepackage{amssymb}
\usepackage{soul}

\usepackage{mathrsfs, mathtools}
\usepackage{overpic}
\usepackage{enumitem}

\let\emph\undefined
\newcommand{\emph}[1]{\textsl{#1}}
\let\itshape\undefined
\let\itshape\slshape

\usepackage{hyperref}

\numberwithin{equation}{section}
\usepackage{mathtools}
\mathtoolsset{showonlyrefs}
\numberwithin{equation}{section}

\newtheoremstyle{style1}
{13pt}
{13pt}
{}
{}
{\normalfont\bfseries}
{.}
{.5em}
{}

\theoremstyle{style1}

\newtheorem{definition}[equation]{Definition}
\newtheorem{example}[equation]{Example}
\newtheorem{remark}[equation]{Remark}

\newcommand{\catf}[1]{{\mathbf{#1}}}

\usepackage{tocloft}

\newtheoremstyle{style2}
{13pt}
{13pt}
{\slshape}
{}
{\normalfont\bfseries}
{.}
{.5em}
{}

\theoremstyle{style2}

\newtheorem{lemma}[equation]{Lemma}
\newtheorem{theorem}[equation]{Theorem}
\newtheorem{proposition}[equation]{Proposition}
\newtheorem{corollary}[equation]{Corollary}

\usepackage{tikz}
\usetikzlibrary{matrix,arrows,decorations.pathmorphing,shapes.geometric}
\usepackage{tikz-cd}

\newcommand{\R}{\mathbb{R}}

\newcommand{\N}{\mathbb{N}}

\newcommand{\St}{\operatorname{St}}
\newcommand{\Un}{\operatorname{Un}}

\usepackage{needspace}

\newcommand{\SO}{\operatorname{SO}}
\newcommand{\Or}{\operatorname{O}}

\newcommand{\Sing}{\operatorname{Sing}}

\newcommand{\holimsub}[1]{\underset{#1}{\operatorname{holim}}\,}
\newcommand{\hocolimsub}[1]{\underset{#1}{\operatorname{hocolim}}\,}
\newcommand{\colimsub}[1]{\underset{#1}{\operatorname{colim}}\,}

\newcommand{\Cospan}{\catf{Cospan}}
\newcommand{\PreCob}{\catf{PreCob}}
\newcommand{\Top}{\catf{Top}}
\newcommand{\PBun}{\catf{PBun}}

\newcommand{\cat}[1]{\mathcal{#1}}

\newcommand{\Int}{\mathbf{Int}}

\newcommand{\End}{\mathsf{End}}
\newcommand{\Hom}{\operatorname{Hom}}

\newcommand{\id}{\text{id}}

\newcommand{\Cob}{{\catf{Cob}}}
\newcommand{\iC}{{(\infty ,1)}}

\newcommand{\DS}{\text{/\hspace{-0.1cm}/}}

\newcommand{\Set}{\catf{Set}}

\let\to\undefined
\newcommand{\to}{\longrightarrow}
\let\mapsto\undefined
\newcommand{\mapsto}{\longmapsto}

\newcommand{\sSet}{\catf{sSet}}
\newcommand{\opp}{\text{opp}}
\let\Bar\undefined
\newcommand{\Bar}{\operatorname{B}}
\newcommand{\EAlg}{E_\infty\text{-}\catf{Alg}}

\DeclareMathSymbol{\Phiit}{\mathalpha}{letters}{"08} 
\DeclareMathSymbol{\Psiit}{\mathalpha}{letters}{"09}
\DeclareMathSymbol{\Sigmait}{\mathalpha}{letters}{"06}
\DeclareMathSymbol{\Xiit}{\mathalpha}{letters}{"04}
\DeclareMathSymbol{\Piit}{\mathalpha}{letters}{"05}\let\Pi\undefined\newcommand{\Pi}{\Piit}
\DeclareMathSymbol{\Gammait}{\mathalpha}{letters}{"00}
\DeclareMathSymbol{\Omegait}{\mathalpha}{letters}{"0A}
\DeclareMathSymbol{\Upsilonit}{\mathalpha}{letters}{"07}
\DeclareMathSymbol{\Thetait}{\mathalpha}{letters}{"02}

\let\Phi\undefined\newcommand{\Phi}{\Phiit}
\let\Sigma\undefined\newcommand{\Sigma}{\Sigmait}
\let\Psi\undefined\newcommand{\Psi}{\Psiit}
\let\Gamma\undefined\newcommand{\Gamma}{\Gammait}

\title{Extended primitive HQFTs}
\author{Notes by Lukas Müller and Lukas Woike}

\definecolor{Blue}  {rgb} {0.282352,0.239215,0.803921}
\definecolor{Green} {rgb} {0.133333,0.545098,0.133333}
\definecolor{Red}   {rgb} {0.803921,0.000000,0.000000}
\definecolor{Violet}{rgb} {0.580392,0.000000,0.827450}

\newcounter{jfc}

\begin{document}

	\begin{flushright}
		\small
		{\sf EMPG--18--17} \\
		{\sf [ZMP-HH/18-20]} \\
		\textsf{Hamburger Beiträge zur Mathematik Nr. 751}\\
		\textsf{September 2018}
	\end{flushright}
	
	\vspace{10mm}
	
	\begin{center}
		\textbf{\LARGE{Equivariant Higher Hochschild Homology and Topological Field Theories}}\\
		\vspace{1cm}
		{\large Lukas Müller $^{a}$} \ \ and \ \ {\large Lukas Woike $^{b}$}
		
		\vspace{5mm}
		
		{\em $^a$ Department of Mathematics\\
			Heriot-Watt University\\
			Colin Maclaurin Building, Riccarton, Edinburgh EH14 4AS, U.K.}\\
		and {\em Maxwell Institute for Mathematical Sciences, Edinburgh, U.K.}\\
		Email: \ {\tt lm78@hw.ac.uk \ }
		\\[7pt]
		{\em $^b$ Fachbereich Mathematik \\ Universit\"at Hamburg\\
			Bereich Algebra und Zahlentheorie\\
			Bundesstra\ss e 55, \ D\,--\,20\,146\, Hamburg }\\
		Email: \ {\tt  lukas.jannik.woike@uni-hamburg.de\ }
	\end{center}
	\vspace{1cm}
	\begin{abstract}
		\noindent We present a version of higher Hochschild homology for spaces equipped with principal bundles for a structure group $G$. 
		As coefficients,
		 we allow $E_\infty$-algebras with $G$-action. 
		For this homology theory, we establish an equivariant version of excision and prove that it extends to an equivariant 
		topological field theory with values in the $(\infty,1)$-category of cospans of $E_\infty$-algebras.
	\end{abstract}

	\tableofcontents
	
	\newpage
	\section{Introduction}
	
	Equivariant topological field theory \cite{turaevhqft} is one of the many different flavors of topological field theory. 
	For a fixed group $G$,
	 a $G$-equivariant topological field theory assigns algebraic quantities to manifolds and bordisms \emph{equipped with a principal $G$-bundle}, which is conveniently modelled by a map into the classifying space $BG$ of the group $G$.
	It turns out that the decoration of the bordisms with $G$-bundles adds new phenomena to the picture and leads to rich algebraic structures. 
	One key aspect responsible for this is the built-in invariance of equivariant topological field theories under the mapping class group action on the groupoid of $G$-bundles over a given manifold. 
	
	So far, the study of equivariant topological field theory concentrates mainly on non-extended field theories with values in vector spaces
	and on the (once extended) 3-2-1-dimensional theories which by evaluation on the circle can be related to equivariant modular categories \cite{turaevhqft, maiernikolausschweigerteq,htv,hrt,schweigertwoikeofk,extofk,MuellerWoike1}. However, in these approaches, higher categorical structures have been mostly left aside. 
	
This article brings together equivariant topological field theories and the powerful homotopy-theoretic techniques that have by now become standard in the study of non-equivariant topological field theory such as complete Segal spaces, factorization homology and factorization algebras, see \cite{ha,af,ginot,cg} for the general background and \cite{lurietft,claudiathesis,benzvibrochierjordan} for the relation to field theories. 
	On the one hand, our motivation is the study of higher analogues of the algebraic structures produced by non-homotopical equivariant field theory; on the other hand we also have a concrete class of examples in mind
	which actually requires a treatment within the framework of homotopy theory, see below.
	
	Thus, we lift the axioms of an equivariant topological field theory to a homotopical setting (Definition~\ref{defhomotopicalequivtft}), by which we mean that we work with $(\infty,1)$-categories.
	This definition is very natural in the sense that the homotopy invariance axiom of ordinary equivariant topological field theories is automatically encoded in a coherent way. The needed equivariant version of the $(\infty,1)$-bordism category (Section~\ref{gbordsec}) will be based on \cite{CalaqueScheimbauer}.
	
	Since already the study of the non-homotopical case teaches us that equivariant topological field theories are highly constrained and hence rare objects, our focus will lie on the production of a first class of examples. 
	It will be provided by a version of higher Hochschild homology of spaces equipped with a principal $G$-bundle.
 As we prove, it gives rise to a homotopical equivariant topological field theory.
	
	The construction of equivariant higher Hochschild homology for a space equipped with a $G$-bundle
	will be based on the most crucial feature of a homology theory:  the local-to-global principle.
It allows to recover the homology of a space by applying a prescribed algebraic procedure to the homology of simpler spaces that our space of interest is built from and thus ensures computability.  
	
	Important instances of local-to-global principles for homology theories include:
	
	\begin{itemize}

	\item Ordinary homology for topological spaces satisfying the Eilenberg-Steenrod axioms, see e.g.\ \cite[IV.6]{bredon}, is determined by its value on the point where it yields a certain Abelian group (referred to as the coefficients). The homology on more complicated spaces can be computed by excision (or the Mayer-Vietoris sequence).  
		
	\item For higher derived Hochschild homology \cite{pirashvili,gtz} the coefficients are $E_\infty$-algebras. Again, this homology theory is determined by the value on the point. To compute it on more complicated spaces one can use the excision property, which means in this case the preservation of homotopy pushouts.

	\item Factorization homology \cite{ha,af} provides an invariant for manifolds. 
	For $n$-dimensional mani\-folds the coefficients are $E_n$-algebras. They prescribe the value of factorization homology on $n$-dimensional disks.
	 Again, an excision property is key for its computation on more complicated manifolds. Factorization homology is intimately related to factorization algebras \cite{cg}.


	\end{itemize}
	The construction principle for all these theories is given roughly as follows: Prescribe the homology theory on a certain class of spaces (ideally, very simple ones that allow to build more complicated spaces: the point, disks, contractible spaces etc.) by using certain algebraic objects, the coefficients (Abelian groups, certain types of algebras etc.). Then the homology theory is extended to all spaces of interest by a suitable (derived enriched) version of left Kan extension. 
	
The coefficients for our equivariant version of higher Hochschild chains are $E_\infty$-algebras with $G$-action. Following the guiding principles outlined above, the equivariant higher Hochschild chains $\int_{\varphi } A$ for an arbitrary map $\varphi : X \to BG$ and an $E_\infty$-algebra with $G$-action as coefficients are computed via left Kan extension. As the name suggests, equivariant higher Hochschild chains generalize ordinary higher Hochschild chains.
	Up to equivalence,
	$\int_{\varphi} A$ only depends on the homotopy class of $\varphi$ or equivalently the isomorphism class of the $G$-bundle classified by $\varphi$ and it carries a representation of the automorphism group of that bundle (Proposition~\ref{eqHHfunctorprop}).  
Moreover, equivariant higher Hochschild homology is essentially constant on mapping class group orbits (Proposition~\ref{propmcgorbits}).

In Section~\ref{mainthm} we state our main result: Equivariant higher Hochschild homology gives rise to a homotopical equivariant topological field theory. 
\\[2ex]
\noindent {\textbf{Theorem~\ref{thmmain}.} \itshape
For a group $G$, let $A$ be a $G$-equivariant $E_\infty$-algebra in chain complexes over a field, i.e.\ an $E_\infty$-algebra with $G$-action.
Then $G$-equivariant higher Hochschild chains 
with coefficients $A$
naturally extend, for any $n\ge 1$, to an $n$-dimensional homotopical equivariant topological field theory, i.e.\
a symmetric monoidal $(\infty,1)$-functor
\begin{align} Z_A = \int_? A : G\text{-}\Cob(n) \to \Cospan (\EAlg)  \label{thehhqft}
\end{align} with values in cospans of $E_\infty$-algebras.} \\[2ex]
Here, the symmetric monoidal target category $\Cospan (\EAlg)$ of cospans of $E_\infty$-algebras is obtained by (dualizing) the results of \cite{haugseng}. 
The target $\Cospan (\EAlg)$ is related to the Morita category of $E_\infty$-algebras, see Example~\ref{excospans}.

A crucial ingredient of the proof of the main Theorem is 
Lurie's $(\infty,1)$-version of the Grothendieck construction \cite[Section~2.2.1 \& 3.2]{htt} 
and an excision property of equivariant higher Hochschild homology (Proposition~\ref{satzexcision}) taking into account not only the gluing of spaces, but also of maps to $BG$. It translates into the needed gluing property of the homotopical equivariant topological field theory.

In Remark~\ref{bordismhypothesisrem} we explain how the existence of the homotopical equivariant topological field theory of Theorem~\ref{thmmain} can also be derived from the cobordism hypothesis.


As an application of our results we compute the equivariant higher Hochschild homology of the circle equipped with a map to $BG$ (Example~\ref{excircle}) and find a $G$-twisted version of ordinary Hochschild homology.
 In Section~\ref{exdw} we make a suggestion on how to describe (equivariant) Dijkgraaf-Witten models in the given setup and justify our choice by orbifoldization arguments. 
 
The present article only treats the $(\infty,1)$-version of equivariant topological field theory, however, not the fully extended version. We avoided the technical complication that would have been associated with the full extension in order to work out more clearly the aspects coming from the decoration with $G$-bundles. 
 For the same reason, we only used a $G$-equivariant version of higher Hochschild homology instead of full-fledged factorization homology.
 Also we concentrated on discrete group rather than topological ones.  
 We hope to come back to these mentioned possibilities of generalization in future work.

\subparagraph{A warning about the nomenclature.}
In \cite{turaevhqft} $G$-equivariant topological field theories are referred to as \emph{homotopy quantum field theories with aspherical target}. Here the word \emph{homotopy} refers to the homotopy invariance property which means that the linear maps assigned to a bordism and a map to $BG$ remains invariant when the map is changed by a homotopy relative boundary.
However, the word homotopy does \emph{not} indicate the use of homotopy-theoretic machinery as we want to apply it in our article. For this reason we have avoided in the above introduction and also in the main text below the word \emph{homotopy quantum field theory with aspherical target} and replaced it with the term \emph{equivariant topological field theory} used also in \cite{maiernikolausschweigerteq,schweigertwoikeofk,extofk,MuellerWoike1}.

	\subparagraph{Acknowledgements.}
We would like to thank 
Marco Benini, 
Adrien Brochier, 
Severin Bunk, 
Tobias Dyckerhoff,
Rune Haugseng, 
Emily Riehl,
Claudia Scheimbauer, 
Alexander Schenkel, 
Christoph Schweigert, 
Richard Szabo, 
Lóránt Szegedy 
and Tim Weelinck 
for helpful discussions.

LM is supported by the Doctoral Training Grant ST/N509099/1 from the UK Science and Technology Facilities
Council (STFC).
LW is supported by the RTG 1670 ``Mathematics inspired by String theory and Quantum
Field Theory'' and thanks the Heriot-Watt University in Edinburgh and in particular Richard Szabo for their hospitality during the time when part of this project was completed.

\needspace{10\baselineskip}
\section{Higher Hochschild homology for spaces with maps into the classifying space of a group}
In this section we define for a given $E_\infty$-algebra $A$ in differential graded vector spaces with action by a group $G$ and a map $\varphi :X \to BG$ from a space (or simplicial set) $X$ to the classifying space of $G$ the equivariant higher Hochschild chains $\int_{\varphi } A$ of $\varphi$ with coefficients in $A$. The equivariant higher Hochschild chains are again an $E_\infty$-algebra in differential graded vector spaces. 

	Denote by $\sSet$ the category of simplicial sets equipped with its usual Quillen model structure, see e.g.\ \cite{goerssjardine}. For any simplicial set $X$ the slice category $\sSet / K$ of maps to a simplicial set $K$ carries a model structure created by the forgetful functor $\sSet / K \to \sSet$ \cite{hirschhorn}. 
	
	For simplicial sets $X$ and $Y$ we denote by $\sSet(X,Y)$ the internal hom. The simplicial enrichment of $\sSet$ induces a simplicial enrichment of $\sSet / K$, see e.g.\ \cite{stevenson}: For $\varphi :X \to K$ and $\psi : Y \to K$ the morphism space $\sSet / K (\varphi :X \to K,\psi : Y \to K )$ is defined as the pullback
	\begin{equation}
		\begin{tikzcd}
		\sSet / K (\varphi :X \to K,\psi : Y \to K ) \ar{rr}{} \ar[swap]{dd}{} & & \sSet (X,Y) \ar{dd}{\psi_*} \\
			& & \\
			\star \ar{rr}{\varphi} & & \sSet(X,K)
		\end{tikzcd}
	\end{equation}
	
	For any group $G$ the classifying space $BG \in \sSet$ comes with a basepoint $\star \to BG$. In the following lemma we compute the derived enriched morphism space $\mathbb{R}	\sSet / BG (\star \to BG,\varphi :X \to BG )$ from $\star \to BG$ to any object $\varphi : X \to BG$ in $\sSet / BG$: 
	
	\begin{lemma}
		For any group $G$ and any map $\varphi : X \to BG$ there is weak equivalence
		\begin{align}
		\mathbb{R}	\sSet / BG (\star \to BG,\varphi :X \to BG ) \simeq \varphi^* EG ,
			\end{align} where $\varphi^* EG$ is the total space of the pullback of the universal $G$-bundle $EG \to BG$ along $\varphi$. 
		\end{lemma}
	
	\begin{proof} All objects in $\sSet / BG$ are cofibrant, hence we obtain by definition
		\begin{align} 	
		\mathbb{R}	\sSet / BG (\star \to BG,\varphi :X \to BG ) \simeq 	\sSet / BG (\star \to BG,\pi : Y \to BG ),
		\end{align} where $\pi : Y \to BG$ is a fibrant replacement of $\varphi : X \to BG$, i.e.\ $\pi : Y \to BG$ is a fibration such that $\pi \circ \iota = \varphi$ for a trivial cofibration $\iota : X \to Y$. But $	\sSet / BG (\star \to BG,\pi : Y \to BG )$ is the pullback of the cospan $\star \to BG \stackrel{\pi}{\longleftarrow} Y$. Since $\pi$ is a fibration and $Y$ is fibrant, this pullback is equivalent to the homotopy pullback \cite[Proposition~A.2.4.4]{htt} and hence, by homotopy invariance of the homotopy pullback, equivalent to the homotopy pullback of the cospan $EG \to BG \stackrel{\varphi}{\longleftarrow} X$. 
		Here we have used that $EG \to \star$ is a trivial fibration. This homotopy pullback is, again by \cite[Proposition~A.2.4.4]{htt}, equivalent to the ordinary pullback since $EG \to BG$ is a fibration.  
		But the pullback of $EG \to BG \stackrel{\varphi}{\longleftarrow} X$ is by definition just $\varphi^* EG$. 
		\end{proof}

	\begin{remark} For any map $\varphi :X \to BG$ the space $\varphi^* EG$ comes with a right $G$-action, i.e.\ it is a functor $\star \DS G^\opp \to \sSet$. Here $\star \DS G$ is the groupoid with one object and automorphism group $G$. Moreover, $G^\opp$ is the group with underlying set $G$ equipped with the opposite multiplication $g \cdot^\opp h := h g$ for $g,h \in G$. \end{remark}

	As the coefficients for the homology theory of spaces with map to a classifying space $BG$ we choose an $E_\infty$-algebra with $G$-action. We refer to \cite{bergerfresse} for background on 
	 the $E_\infty$-operad and $E_\infty$-algebras in chain complexes.
	
	\begin{definition}
	For a field $k$ (that we fix through this article) we denote by $\EAlg$ the category of $E_\infty$-algebras in differential graded vector spaces over $k$. For a group $G$ we define a \emph{$G$-equivariant $E_\infty$-algebra} as a functor $\star \DS G \to \EAlg$ from the groupoid with one object and automorphism group $G$ to $\EAlg$
	(i.e.\ as an $E_\infty$-algebra with $G$-action). The category of $G$-equivariant $E_\infty$-algebras will be denoted by $\EAlg^G$. 
	\end{definition}

\begin{remark}
	Instead of $E_\infty$-algebras in differential graded vector spaces, one could consider $E_\infty$-algebras in some other symmetric monoidal $(\infty,1)$-category such as differential graded modules over a commutative ring.
	 In this case, however, the tensor product might not be exact leading to technical difficulties that we would like to avoid.
	\end{remark}
	
	Now for any $G$-equivariant $E_\infty$-algebra $A$ we can consider the left derived functor tensor product
	\begin{align}
	\mathbb{R}	\sSet / BG (\star \to BG,\varphi :X \to BG ) \stackrel{\mathbb{L}}{\otimes}_G A \simeq	\varphi^* EG \stackrel{\mathbb{L}}{\otimes}_G A \label{eqndefeqHH}
		\end{align} which is modeled by the bar construction \cite[Chapter~4 \& 10]{riehl}. 
		We can see this as a derived enriched left Kan extension 
		of $A$ (seen as functor from the full subcategory of $\sSet/BG$ on $\star \to BG$ to $\EAlg$)
		along the inclusion of $\star \to BG$ into $\sSet / BG$. 
		Explicitly, \eqref{eqndefeqHH} is given by the geometric realization of the simplicial $E_\infty$-algebra $B_*(\varphi^* EG,\star\DS G,QA)$ with $n$-simplices given by
	\begin{align}
	B_n(\varphi^* EG,\star\DS G,QA) = \varphi^* EG \times G^n \otimes QA,
		\end{align} where $\otimes$ denotes the $\sSet$-tensoring of $\EAlg$ \cite{fresseco} and $QA$ is a cofibrant replacement of $A$ as an $E_\infty$-algebra (here we use the projective model structure on differential graded vector spaces and the induced model structure on $E_\infty$-algebras). 
	In summary,
	\begin{align}
		\varphi^* EG \stackrel{\mathbb{L}}{\otimes}_G A = \int^n  \left(  \left(  \varphi^*EG \times \Delta_n \times G^n \right) \otimes QA \right) \in \EAlg. \label{defeqnZA}
		\end{align}
		
		\begin{definition}\label{defeqHH}
			For a $G$-equivariant $E_\infty$-algebra $A$ and a map $\varphi : X \to BG$ we define the \emph{$G$-equivariant higher Hochschild chains of $\varphi$ with coefficients $A$} as
			\begin{align} \int_{\varphi } A :=\varphi^* EG \stackrel{\mathbb{L}}{\otimes}_G A \in \EAlg.\end{align}
			\end{definition}

\begin{remark}	For $G=1$ Definition~\ref{defeqHH} reduces to the ordinary (higher) Hochschild homology $CH(X;A) = X \stackrel{\mathbb{L}}{\otimes} A$ of a space $X$ with coefficients in an $E_\infty$-algebra $A$ \cite{pirashvili,gtz}. \end{remark}
	
	\begin{lemma}\label{lemmatrivialbundle}
		For any group $G$ and $A \in \EAlg^G$ there is a canonical weak equivalence
		\begin{align}
			\int_{\varphi } A \simeq X \stackrel{\mathbb{L}}{\otimes} A\ ,
			\end{align} whenever the $G$-bundle classified by $\varphi$ is trivial.
		\end{lemma}
	
	\begin{proof}
		If $\varphi$ classifies a trivial $G$-bundle, 
		we have $\varphi^* EG \cong X \times G$. 
		Now the $G$-action on $QA$ provides an augmentation of $B_*(\varphi^* EG,\star\DS G,QA)$ by $X \otimes QA$. 
		Using the unit element of $G$ one finds extra degeneracies \cite[Section~4.5]{riehl} for this augmentation, 
		which proves $\int_{\varphi } A \simeq X \otimes QA = X \stackrel{\mathbb{L}}{\otimes} A$. 
		\end{proof}
	
	\begin{corollary}
		For any group $G$ and $A \in \EAlg^G$ the value of equivariant higher Hochschild homology on the point $\star \to BG$ is given by $A$, as an $E_\infty$-algebra. 
		\end{corollary}

		For later use we record the following technical Lemma: 
		
		\begin{lemma}\label{lemmacofbetcofobj}
		Let $G$ be a group and $A \in \EAlg^G$.
		For any cofibration from $\varphi :X \to BG$ to $\psi : Y \to BG$ in $\sSet / BG$ the induced map
		\begin{align}
			\int_{\varphi } A \to 	\int_{\psi } A
		\end{align} is a cofibration between cofibrant $E_\infty$-algebras. 
		\end{lemma}

		\begin{proof}
		\begin{enumerate}
		
		\item For $\varphi :X \to BG$ the $n$-th latching object $L_n B_*(\varphi^* EG,\star\DS G,QA)$ of the simplicial bar construction $B_*(\varphi^* EG,\star\DS G,QA)$ is given by
						\begin{align}
						L_n B_*(\varphi^* EG,\star\DS G,QA) =  \varphi^* EG \times D_nG \otimes QA ,
						\end{align} where $D_n G \subset B_n G=G^n$ is the set of degenerate $n$-simplices of $BG$, see \cite[(proof of) Lemma~5.2.1]{riehl}. The inclusion $D_n G \subset B_n G=G^n$ induces the $n$-th latching map
						\begin{align}
							L_n B_*(\varphi^* EG,\star\DS G,QA) =  \varphi^* EG \times D_nG \otimes QA \to B_n(\varphi^* EG,\star\DS G,QA)=\varphi^* EG \times G^n \otimes QA
						\end{align} which is a cofibration since $? \otimes QA$ is left Quillen. Hence, $B_*(\varphi^* EG,\star\DS G,QA)$ is Reedy cofibrant, see also \cite[Remark 5.2.2]{riehl}. Since the geometric realization is a left Quillen functor on the category of simplicial objects equipped with the Reedy model structure \cite[Proposition~VII.3.6]{goerssjardine}, this implies that $\int_{\varphi } A$ is cofibrant.

						\item In order to show that the map $\int_{\varphi } A \to \int_{\psi } A$ is a cofibration, we use again that the geometric realization is left Quillen. Hence, it suffices to show that \begin{align} B_*(\varphi^* EG,\star\DS G,QA) \to B_*(\psi^* EG,\star\DS G,QA)\end{align} is a Reedy cofibration, i.e.\ that the maps \small
									\begin{align}
										B_n(\varphi^* EG,\star\DS G,QA) \cup_{L_n B_*(\varphi^* EG,\star\DS G,QA)} L_n B_*(\psi^* EG,\star\DS G,QA) \to B_n(\psi^* EG,\star \DS G,QA) \label{relativlatchingmapeqn}
										\end{align}\normalsize are cofibrations. 
											First note that
														\begin{align}
													&	B_n(\varphi^* EG,\star \DS G,QA) \cup_{L_n B_*(\varphi^* EG,\star \DS G,QA)} L_n B_*(\psi^* EG,\star \DS G,QA)\\ \cong &  \left((\varphi^* EG \times G^n) \cup _{\varphi^* EG \times D_n G }(\psi^* EG \times D_n G) \right) \otimes QA
														\end{align} since $? \otimes QA$ preserves pushouts. Under this isomorphism the map \eqref{relativlatchingmapeqn} is given by the image of
														\begin{align}
																		(\varphi^* EG \times G^n) \cup _{\varphi^* EG \times D_n G }(\psi^* EG \times D_n G) \to \psi^* EG \times G^n \label{eqnsm7eqn}
																		\end{align} by $? \otimes QA$. Since $? \otimes QA$ preserves cofibrations, it suffices that \eqref{eqnsm7eqn} is a cofibration, which follows from the fact that $\sSet$ is a simplicial model category. \end{enumerate}\end{proof}

		A homotopy $\varphi_0 \simeq \varphi_1 : X \to BG$ induces an isomorphism $\varphi_0^* EG \cong \varphi_1^* EG$ of $G$-bundles only depending on the equivalence class of the homotopy. This implies immediately: 
	
	\begin{lemma}\label{lemmahtpmaps}
		A homotopy $\varphi_0 \simeq \varphi_1 : X \to BG$ induces an isomorphism \begin{align} \int_{\varphi_0} A \cong \int_{\varphi_1 } A\end{align} for any $G$-equivariant $E_\infty$-algebra $A$. This isomorphism only depends on the equivalence class of the homotopy.
	\end{lemma}
	
	For a space $X$ denote by $BG^X$ the space of maps from $X$ to $BG$ and by $\Pi(X,BG) := \Pi (BG^X)$ the fundamental groupoid of this mapping space. There is a canonical equivalence from $\Pi(X,BG)$ to the groupoid $\PBun_G(X)$ of principal $G$-bundles over $X$
	\cite{heinlothstacks,extofk}. From Lemma~\ref{lemmahtpmaps} we now conclude:
	
	\begin{proposition}\label{eqHHfunctorprop}
		The equivariant higher Hochschild chains with coefficients in a $G$-equivariant $E_\infty$-algebra $A$ provide a functor
		$\PBun_G(X) \to \EAlg$. 
	\end{proposition}
	
	We will use the notation $\int_P A := \int_{\varphi } A$ for the principal bundle $P:=\varphi^* EG$; equivalently, we can define $\int_P A := P \stackrel{\mathbb{L}}{\otimes}_G A$. 
	
	According to this Proposition, $\int_? A$ is essentially constant on an isomorphism class in $\PBun_G(X)$. In fact, it is even essentially constant on a mapping class group orbit in $\PBun_G(X)$:
	
	\begin{proposition}\label{propmcgorbits}
		Any isomorphism $\Phi : X \to Y$ of spaces induces an isomorphism
		\begin{align} \int_{\Phi^* P} A \cong \int_{P} A\end{align} for any $G$-bundle $P$ over $Y$ and any $G$-equivariant $E_\infty$-algebra $A$. 
	\end{proposition}

\begin{proof}
	Recall that
	\begin{align} \Phi^* P = \{ (x,p) \in X \times P \, |\, p \in P_{\Phi(x)} \} .
	\end{align} Projection to the second factor yields a canonical $G$-equivariant map $\Phi^* P \to P$. By $\Phi^{-1}(p) := (\Phi^{-1}(x),p)$ for $p \in P_x$ and $x\in X$ we obtain its inverse, so $\Phi^* P \cong P$ by a canonical isomorphism of right $G$-spaces (equivalently: a canonical natural isomorphism of functors $\star \DS G^\opp \to \sSet$). This is not an isomorphism of $G$-bundles, but still enough to conclude $\Phi^* P \stackrel{\mathbb{L}}{\otimes}_G A \cong P \stackrel{\mathbb{L}}{\otimes}_G A$. 
	\end{proof}

\begin{remark}[Relation to equivariant factorization homology]
Let $A$ be an  $E_\infty$-algebra with $G$-action.
In Definition~\ref{defeqHH} we have defined equivariant higher Hochschild homology of any space with free right $G$-action (i.e.\ a $G$-bundle) by $X \stackrel{\mathbb{L}}{\otimes}_G A$. In fact, the same formula makes sense also for spaces with arbitrary right $G$-actions, which allows us to extend equivariant higher Hochschild homology to those spaces.
This generalization allows us to briefly comment on the relation of equivariant higher Hochschild homology
to
an equivariant version of factorization homology that was introduced in \cite{EFH} after a first draft of this
paper appeared as a preprint. 

First we summarize the notion of equivariant factorization homology from \cite{EFH}:
Let $M$ be an $n$-dimensional manifold equipped with the action of a finite group 
$G$ and $\rho : G \longrightarrow \End(\R^n)$ a representation of $G$. 
A \emph{$G$-equivariant $\rho$-framing} on $M$ is an isomorphism $TM \cong M\times \R^n$
satisfying $g_*(m,v)=(gm,\rho(g)v)$ for $m\in M$ and $v\in T_m M$, where $g_*$ denotes the pushforward along the
action of $g$. 
Manifolds of dimension $n$
with $G$-equivariant $\rho$-framing
form a symmetric monoidal $(\infty,1)$-category
$G^\rho\mathbf{Orb}^{\text{fr}}_n$
under disjoint union;
we refer to \cite[Section~2]{EFH} for details. 
Let $I$ be a subgroup of $G$ such that the representation $\rho$ restricted to 
$I$ is faithful. We denote by $\mathbb{D}_I$ the space 
$G\times_I\R^n= G\times\R^n /{\sim}$, where $\sim$ identifies  $(g,\rho(i)x)$ with $(g\cdot i,x )$ for $i\in I$, $g\in G$ and $x\in \R^n$.  
This space is equipped with the obvious left $G$-action and $\rho$-framing induced from the standard 
framing on $\R^n$. 
The resulting object in $G^\rho\mathbf{Orb}^{\text{fr}}_n$ is called a \emph{local model}.
We denote by $G^\rho\mathbf{Disk}^{\text{fr}}_n$ the full subcategory
of $G^\rho\mathbf{Orb}^{\text{fr}}_n$ whose objects are disjoint unions of local models 
$\mathbb{D}_I$. A \emph{$G^\rho\mathbf{Disk}^{\text{fr}}_n$-disk algebra} in a symmetric
monoidal $(\infty,1)$-category $\mathcal{C}$ is a symmetric monoidal functor
\begin{align}
O : G^\rho\mathbf{Disk}^{\text{fr}}_n \longrightarrow \mathcal{C} \ \ .
\end{align}  
Equivariant factorization homology of $O$
\begin{align}
\int_- O : G^\rho\mathbf{Orb}^{\text{fr}}_n \to \cat{C}
\end{align}
 is now defined \cite[Definition~4.14]{EFH} via left Kan extension
\begin{equation}
\begin{tikzcd}
 G^\rho\mathbf{Disk}^{\text{fr}}_n \ar{r}{O} \ar{d} & \mathcal{C} \\ 
G^\rho\mathbf{Orb}^{\text{fr}}_n \ar[swap]{ru}{\int_- O } 
\end{tikzcd} \ \ .
\end{equation}

To compare equivariant higher Hochschild homology (which is defined for right actions) to equivariant factorization homology (which is defined for left actions),
we use that a left action on a space $X$ can be turned into a right action by acting with the inverse; we denote $X$ with this right action by $X^\opp$.

If we are given now an $E_\infty$-algebra $A$ with $G$-action,
 we can compute equivariant higher Hochschild homology with coefficients in $A$ for the local models $\mathbb{D}_I$ and compose with the forgetful
functor from $E_\infty$-algebras to chain complexes to get a $G^\rho\mathsf{Disk}^{\text{fr}}_n$-disk algebra $\widehat{A}$ in chain complexes.
This construction gives rise to an $\infty$-functor $\widehat{-}: \EAlg^G  \longrightarrow \catf{Alg}_{G^\rho\mathsf{Disk}^{\text{fr}}_n}(\catf{Ch})$.
It follows from a generalization of Proposition~\ref{satzexcision} to not necessarily free actions
and \cite[Theorem 4.33]{EFH} that equivariant higher Hochschild homology of $A$ and 
equivariant factorization homology of $\widehat{A}$ agree in the sense that
\begin{align}
\int_{\hat{A}} X \cong X^\opp \stackrel{\mathbb{L}}{\otimes}_G A \ \ .
\end{align}
From \cite[Proposition 4.6]{EFH} it follows that we can also compute 
equivariant Hochschild homology via factorization homology~\cite{af} for 
$\mathsf{Disk}_G$-algebras.    

To compute the disk algebra $\widehat{A}$ explicitly,
observe that $\mathbb{D}_I$ is 
homotopy equivalent as a $G$-space to the discrete space $G/I$. We denote by $\iota
: I \longrightarrow G$ the inclusion of $I$ into $G$.
From the homotopy invariance of equivariant Hochschild homology we obtain
\begin{align}
\mathbb{D}_I^\opp\stackrel{\mathbb{L}}{\otimes}_G A \simeq G/I^\opp \stackrel{\mathbb{L}}{\otimes}_G A \simeq B(*/\iota)\stackrel{\mathbb{L}}{\otimes}_G A \simeq \operatorname{hocolim} (I \overset{\iota}{\hookrightarrow}G\rightarrow \EAlg)  \ \ ,
\end{align} 
where the homotopy colimit is computed in $\EAlg$ and the last weak equivalence follows from \cite[Proof of Theorem 8.5.6]{riehl}. This homotopy colimit describes the homotopy
coinvariants $A_I$ of $A$ with respect to the induced action of the subgroup $I$. We conclude that the disk algebra $\widehat A$ that $A$ gives rise to can be explicitly described by 
$\widehat A(\mathbb{D}_I) \simeq A_I$.
\end{remark}

\section{Construction of homotopical equivariant topological field theories}
		In this main section of the article we define the $(\infty ,1)$-category of $G$-cobordisms based on \cite{gtmw,lurietft,CalaqueScheimbauer} allowing us to define homotopical equivariant topological field theories.
		Afterwards, we prove that equivariant higher Hochschild chains (Definition~\ref{defeqHH}) give rise to an example of such a theory. 
		
		\subsection{The $(\infty ,1)$-category of $G$-cobordisms \label{gbordsec}}
		We will model the $(\infty,1)$-category $G\text{-}\Cob(n)$ of $n$-dimensional $G$-cobordisms following \cite{lurietft,CalaqueScheimbauer} using (complete) Segal spaces introduced in \cite{rezk}. 
		
	Let us briefly recall the most important notions: 
		\begin{definition}
			A \textit{Segal space} is a simplicial space $X_\bullet $, i.e. simplicial Kan complex, such that
			\begin{equation}\label{Eq: Segal condition}
			\begin{tikzcd}
			X_{n+m} \ar{r} \ar{d}& X_m \ar{d} \\
			X_n \ar{r} & X_0
			\end{tikzcd}
			\end{equation}
is a homotopy pullback square. The morphisms are induced by the 
morphism 

\begin{align}\begin{array}{rclrcl}
[n]& \to &[n+m], & i &\mapsto &i, \\
\text{$[m]$}& \to &[n+m], & i &\mapsto& i+n \\
\text{$[0]$} & \longrightarrow& [n], & 0 &\longmapsto& n \\ 
\text{ $ [0] $  } & \longrightarrow& [m] ,&   0 &\longmapsto &0 \\
\end{array}
\end{align} 
in the simplex category $\Delta$.
		\end{definition} 	
Note that following \cite{lurietft} we do not require $X$ to Reedy cofibrant as opposed to \cite{rezk}.
		
		Given a Segal space $X_\bullet$ we can build its homotopy category $hX_\bullet$: The objects are the vertices of $X_0$. The set $\Hom_{hX_\bullet}(x,y)$ of morphisms from $x$ to $y$ is given by the set of path connected components of the iterated homotopy fibre product
		\begin{align}
		\{x \} \times^h_{X_0} X_1 \times^h_{X_{0}} \{y\} \ . 
		\end{align}
		The composition of morphisms is induced by
		\begin{align}
		\left( \{x \} \times^h_{X_0} X_1 \times^h_{X_{0}} \{y\} \right) \times \left( \{y \} \times^h_{X_0} X_1 \times^h_{X_{0}} \{z\} \right) & \hookrightarrow  \{x \} \times^h_{X_0} X_1 \times^h_{X_{0}} X_1 \times^h_{X_{0}} \{z\} \\
		&\simeq \{x \} \times^h_{X_0} X_2 \times^h_{X_{0}} \{z\} \\
		&\rightarrow \{x \} \times^h_{X_0} X_1 \times^h_{X_{0}} \{z\} \ .
		\end{align}
		where the appearing weak equivalence is a consequence of \eqref{Eq: Segal condition} and the last morphism is induced by the first face map $\partial_1: X_2\longrightarrow X_1$.
We call an element $f\in X_1$ \textit{invertible} if it is invertible seen as a morphism in $hX_\bullet$. 
Next one defines a space $Z$ of invertible elements together with a map $Z\longrightarrow X_1$\cite[Section 1.4.2]{CalaqueScheimbauer}.
For a Reedy cofibrant Segal space $Z$ is just the subspace of invertible elements.
However, in the general case the definition of the space of invertible elements is more involved, and we refer to
\cite[Section 1.4.2]{CalaqueScheimbauer} for details. 
The degeneracy map $\delta : X_0 \to X_1$ factors through $Z$.
		
		\begin{definition}
			A \textit{complete Segal space} is a Segal space $X_\bullet$ such that $\delta :  X_0 \to Z$ is a homotopy equivalence. 
		\end{definition}
		
		\begin{remark}
			Given an $(\infty,1)$-category described by a complete Segal space $X_\bullet$ we think of the topological space $X_0$ as the $(\infty,0)$-category describing the maximal subgroupoid of $X$. More generally we think of $X_k$ as the $(\infty,0)$-category with objects being collections of $k$ composable morphisms. The completeness condition then implies that the invertible morphisms in $hX$ are exactly the paths in $X_0$ as it should be for this intuition to be consistent.     
		\end{remark}  
		
		To define symmetric monoidal $\iC$-categories, one introduces the category $\Gamma$ with 
		 finite pointed sets $\langle m \rangle= \lbrace *=0,1,\dots , m \rbrace$ 
		 as objects
		 and pointed maps as morphisms.
		For $1\le i\le m$ the \textit{Segal map} $\gamma_i: \langle m \rangle \rightarrow \langle 1 \rangle$
		is defined by
		\begin{align}
	 \langle m \rangle  \ni 	j \longmapsto \delta_{ij} \in \langle 1 \rangle \ .
		\end{align}	
		With this notation we can give the following definition of a symmetric monoidal $\iC$-category
		\cite{Toe, TV15}
		\begin{definition}
			A \textit{symmetric monoidal complete Segal space} is a functor 
			\begin{align}
			A : \Gamma \longrightarrow \text{Complete Segal spaces}
			\end{align}
			such that for ever $m\ge  0$ the map 
			\begin{align}
			\prod_{1\leq i\leq m} A\left(\gamma_i\right): A(\langle m \rangle)\rightarrow A(\langle 1 \rangle)^m
			\end{align}
			is an equivalence of complete Segal spaces.
			Such a functor is called \emph{$\Gamma$-object (in complete Segal spaces)}.
		\end{definition}

\begin{remark}
Form a conceptional point of view, a weaker definition 
in terms of $(\infty,1)$-functors
would be more
appropriate. However, due to the 
strictification result of \cite{TV02} the simpler definition suffices. 
\end{remark}
	
			\begin{definition}
		A \textit{symmetric monidal structure on an $\iC$-category} $X$ is a symmetric monoidal complete Segal space $A$ 
		such that
		 $A(\langle1\rangle)=X$. 
		A \textit{(strict) symmetric monoidal $\iC$-functor} is a natural transformation between the underlying functors 
		\begin{align}
		\Gamma \to \text{Complete Segal spaces}\ . 
		\end{align}
		\end{definition}

		Having recalled these notions we now begin with the definition of the symmetric monoidal $(\infty,1)$-category $G\text{-}\Cob(n)$ of $n$-dimensional $G$-cobordisms:
		The manifolds that this category is built from are equipped with an orientation and a map to the classifying space $BG$ of $G$, where $G$ is always seen as a discrete topological group. We will see this decoration as an $\SO(n)\times G$-structure: For the universal $\SO(n)\times G$-bundle $E(\SO(n)\times G) \to B(\SO(n) \times G)$ consider the associated rank $n$ vector bundle $\xi_G  := E(\SO(n)\times G) \times_{\SO(n)\times G } \mathbb{R}^n$, where we use the group morphism
		\begin{align}
		\SO(n)\times G \xrightarrow{\operatorname{pr}_{\SO(n)}} \SO(n)\longrightarrow \Or(n) \label{Eq: Grouphomo}
		\end{align}
		to define the needed action of $\SO(n)\times G$ on $\mathbb{R}^n$. 
		Now an $\SO(n)\times G$-structure on a manifold $M$ (with boundary) consists of a smooth map $F :  M\to B(\SO(n)\times G)$ 
		together with an isomorphism $ F^* \xi_G  \cong TM$ from  $F^* \xi_G  $ to the tangent bundle of $M$.
		
Since $B(\SO(n)\times G)\cong B\SO(n)\times B G$, we can describe $F$ by its components $f:  M \to B\SO(n)$ and $\psi : M \to BG$;
 and since \eqref{Eq: Grouphomo} factors through $\SO(n)$, an isomorphism $F^* \xi_G \cong TM$ amounts to an isomorphism $f^* \xi \to TM$, where $\xi := E(\SO(n)) \times_{\SO(n)} \mathbb{R}^n$. This is equivalent to a reduction of the frame bundle to $\SO(n)$, i.e.\ an orientation. The map $\psi$ is the classifying map for a principal $G$-bundle. 
In summary, an $\SO(n)\times G$-structure on a manifold $M$ amounts to an orientation of $M$ and a principal $G$-bundle on $M$.

		For the definition of $G\text{-}\Cob(n)$ we recall from \cite[Section 4]{CalaqueScheimbauer} the complete Segal space $\Int_\bullet$ with
	\begin{align}
		\Int_k := &\bigg
		\{ \underline{I}=(\underline{a},\underline{b})= (a_0,\dots ,a_k,b_0, \dots , b_k)\in \R^{2(k+1)}  \\ & \text{ with } a_0\leq a_1 \leq \dots \leq a_k \ , \ b_0\leq b_1 \leq \dots \leq b_k \text{ and } a_j<b_j\bigg\}
	\end{align} for $k \ge 0$. 
		We equip $\Int_k\subset \R^{2(k+1)}$ with the subspace topology. An element in $\Int_k$ can be thought of as $k+1$ ordered intervals $I_0= (a_0,b_0] \leq I_1=[a_1,b_1] \leq \dots \leq I_k= [a_k,b_k)$ closed in $(a_0,b_k)$. 
		A morphism $f :  [m]\longrightarrow [k]$ in the simplex category $\Delta$ acts by the map
		\begin{align}
		\Int_k &\longrightarrow \Int_m \\
		I_0 \leq \dots \leq I_k &\longmapsto I_{f(0)} \leq \dots \leq I_{f(m)} 
		\end{align} and turns $\Int_\bullet$ into a simplicial space which is in fact  a complete Segal space by \cite[Proposition 4.7]{CalaqueScheimbauer}. 

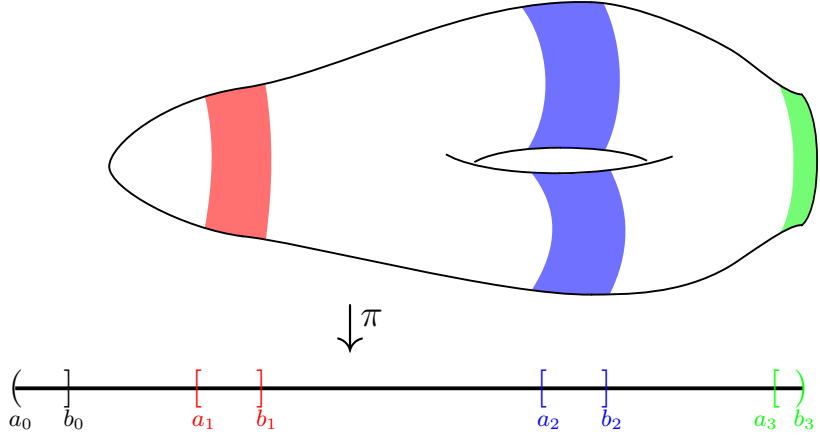
\begin{figure}
\begin{center}
\begin{overpic}[scale=0.45
,tics=10]
{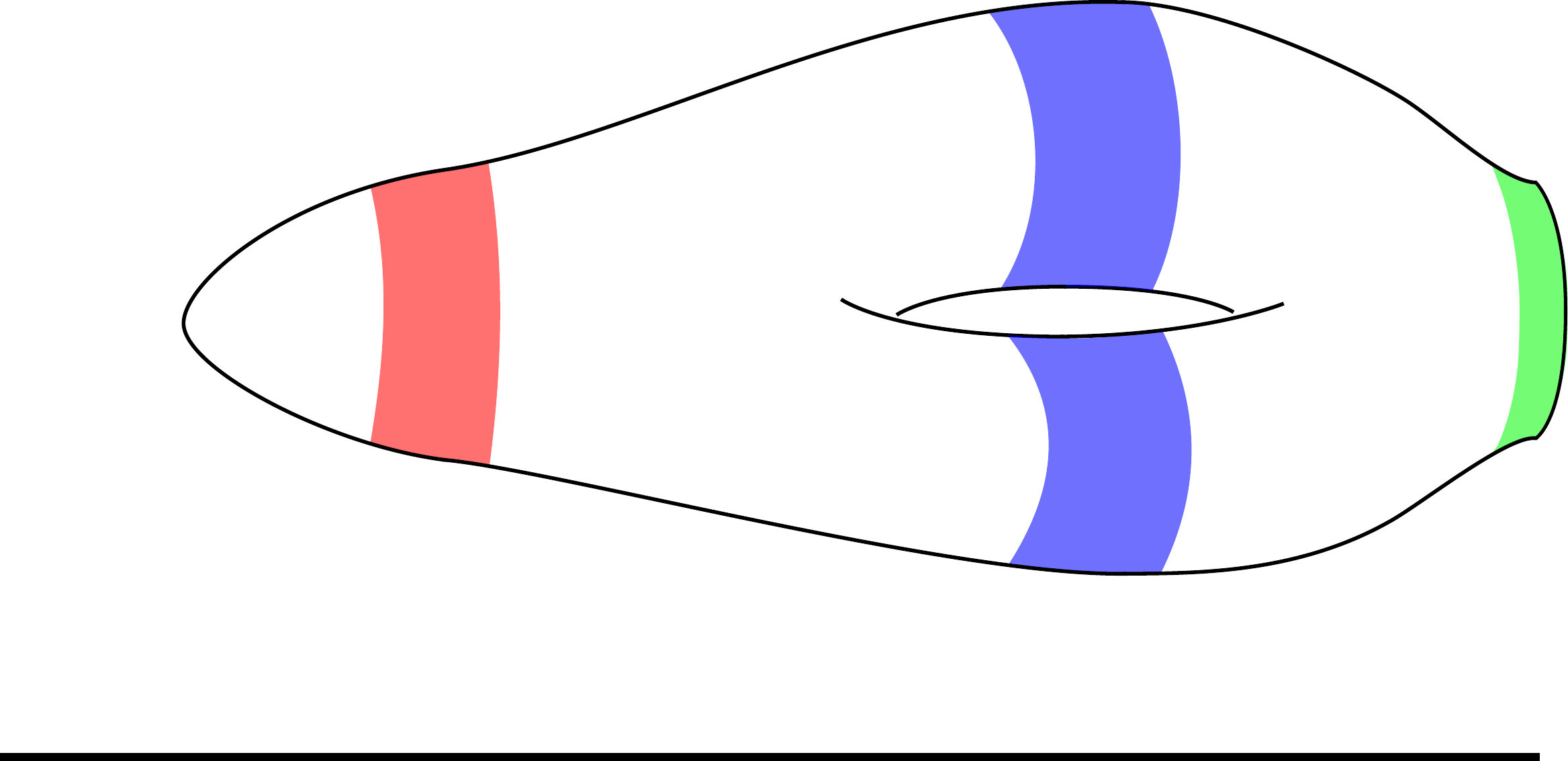}
\put(-0.8,-1){\Large$($}
\put(6,-1){\Large$]$}
\put(22,-1){\color{red}\Large$[$}
\put(30,-1){\color{red}\Large$]$}
\put(65,-1){\color{blue}\Large$[$}
\put(73,-1){\color{blue}\Large$]$}
\put(94,-1){\color{green}\Large$[$}
\put(97,-1){\color{green}\Large$)$}
\put(40,6){\huge$\downarrow$}
\put(43,8){\Large $\pi $}
\put(-0.8,-4.5){\small$a_0$}
\put(6,-4.5){\small$b_0$}
\put(22,-4.5){\color{red}\small$a_1$}
\put(30,-4.5){\color{red}\small$b_1$}
\put(65,-4.5){\color{blue}\small$a_2$}
\put(73,-4.5){\color{blue}\small$b_2$}
\put(92,-4.5){\color{green}\small$a_3$}
\put(97,-4.5){\color{green}\small$b_3$}
\end{overpic}
\vspace{0.5cm}
\caption{An example for the manifold underlying an element of $G\text{-}\PreCob_3^V(n)$.}
\end{center}
\end{figure}
		
		In the next step define the Segal space $G\text{-}\PreCob_\bullet^V(n)$ depending on a finite dimensional real vector space $V$. 
		The space $G\text{-}\PreCob_k^V(n)$ consists of 5-tuples 
		\begin{align}
		(M,f :  M \to B\SO(n),\psi:  M \to BG, \alpha :  f^*\xi \cong TM  ,\underline{I}) \label{eqn5tupel}
		\end{align}
		where $\underline{I}$ is an element of $\Int_k$ and $M$ is a (not necessarily compact) closed
		bounded submanifold of $V\times (a_0,b_k)$ such that
		\begin{itemize}
			\item 
			$\pi :  M \to V\times (a_0,b_k) \to (a_0,b_k)$ is a proper map, 
			
			\item 
$\pi$ is a submersion when restricted to the preimage of the intervals $\underline{I} =(I_0,I_1,\dots, I_k)$.
		\end{itemize} Moreover, $(f,\psi,\alpha)$ is an $\SO(n)\times G$-structure on $M$.
		
The simplicial structure is induced from the structure of $\Int_\bullet$:
A morphism $g:  [m] \to [k]$ in $\Delta$ is sent to the morphism
\small
\begin{align}
G\text{-}\PreCob_k^V(n) & \longrightarrow G\text{-}\PreCob_m^V(n) \\ 
\left(M,f,\psi,\alpha, (a_0, \dots ,b_k)\right) &\mapsto \left(g^*M := \pi^{-1}(a_{g(0)},b_{g(m)}), f|_{g^*M},\psi|_{g^*M}, \alpha|_{g^*M}, (a_{g(0)}, \dots ,b_{g(m)}) \right) \ .\label{Eq: Simplicial structure for GCob}
\end{align} \normalsize
In order to equip $ G\text{-}\PreCob_k^V(n)$ with a topology we construct a Kan complex $G\text{-}\PreCob_k^V(n)_\bullet$ with $G\text{-}\PreCob_k^V(n)$ as 0-simplices. An $\ell$-simplex consists of a 5-tuple (compare to \cite[Definition 9.10]{CalaqueScheimbauer})

\begin{align}
(M,f :  M \to B\SO(n),\psi:  M \to BG, \alpha  ,(a_0(s),\dots , b_k(s))_{s\in |\Delta_\ell| }) \ , \label{anellsimplex}
\end{align}  where $|\Delta_\ell|= \{ (x_0,\dots , x_\ell)\in \R^{l+1} \mid \sum_{i=0}^\ell x_i =1 \}$ is the geometric standard simplex,
such that
\begin{itemize}
\item 
$(a_0(s)\dots , b_k(s))$ is an element of $\Int_k$ for all $s\in |\Delta_\ell|$ and
\begin{align}
|\Delta_\ell| \longrightarrow \R^{2(k+1)} , \quad  s \mapsto (a_0(s), \dots b_k(s))
\end{align}
is a smooth map,

\item $M$ is a $n+\ell$-dimensional submanifold of \begin{align} V\times (a_0(\cdot ),b_k(\cdot )) \times |\Delta_\ell|:=\lbrace (v,\lambda,s)\in V\times \R \times |\Delta_\ell| \, |\,  a_0(s) < \lambda < b_k(s) \rbrace \, \end{align}

\item 
the map $\pi : M \to V\times (a_0(\cdot ),b_k(\cdot )) \times |\Delta_\ell| \to (a_0(\cdot ),b_k(\cdot )) \times |\Delta_\ell|$ is proper,

\item 
the map $\pi_\ell    :     M \to V\times (a_0(\cdot ),b_k(\cdot )) \times |\Delta_\ell| \to  |\Delta_\ell|$ is a submersion,

\item 
$\alpha    :     \ker (D\pi_\ell    :     TM \rightarrow T|\Delta_\ell|) \rightarrow f^*\xi$ is a fibre-wise linear isomorphism and

\item
$\pi$ is a submersion when restricted to the preimage $\pi^{-1}((a_i(\cdot),b_i(\cdot))\times |\Delta_\ell|)$ for all $i\in \{ 0,\dots , k \}$.

\end{itemize}
One can show that $M \to|\Delta_\ell|$ is a fibre bundle~\cite[Remark 5.9]{CalaqueScheimbauer}. 
For a morphism $g    :     [m]\to [\ell]\in \Delta$ we denote by $|g|   :     |\Delta_m|\to |\Delta_\ell|$ the induced map on standard simplices. The corresponding map
\begin{align}
g^\Delta : G\text{-}\PreCob_k^V(n)_\ell \longrightarrow G\text{-}\PreCob_k^V(n)_m
\end{align}
sends $(a_0(s)\dots b_k(s))_{s\in |\Delta_\ell|}$ to $(a_0(|g|(s))\dots b_k(|g|(s)))_{s\in |\Delta_m|}$, $M$ to $|g|^*M$ and the additional structure also to its pullback along $|g|$.
		
In order to make the resulting simplicial space independent of $V$, we pick an infinite dimensional vector space $\R^\infty$ and define $G\text{-}\PreCob(n)$ as the homotopy colimit 
\begin{align}
G\text{-}\PreCob(n)= \underset{V\subset \R^\infty }{\operatorname{hocolim}}\, G\text{-}\PreCob^V(n) 
\end{align}	
running over all finite-dimensional subspaces of $\mathbb{R}^\infty$. 
This is a Segal space by \cite[Section 9]{CalaqueScheimbauer}, however it is not complete in general. Therefore, one is led to the following definition:
		
\begin{definition}
For any group $G$ we define the \emph{$(\infty,1)$-category of $n$-dimensional $G$-cobordisms $G\text{-}\Cob(n)$} as the complete Segal space obtained by completion of the Segal space $G\text{-}\PreCob(n)$. 
\end{definition}

To define the symmetric monoidal structure on $G\text{-}\Cob(n)$ it is enough to define a $\Gamma$-object 
in Segal spaces $A$ with $A(\langle 1 \rangle)=G\text{-}\PreCob(n)$, since
the image of a Segal space under the completion map is weakly equivalent to the original space and finite products are preserved up to weak equivalences.
For a finite dimensional vector space $V$ and $m \geq 0$ let $G\text{-}\PreCob^V(n)[m]$ be the Segal space which in level $k$ is given by all
\begin{align}\label{Eq: Element Gamma Space}
\left( M_1,\dots ,M_m; f_1,\dots ,f_m; \psi_1,\dots ,\psi_m; \varphi_1,\dots ,\varphi_m; (a_0, \dots b_k) \right)
\end{align}
such that $\left( M_i, f_i, \psi_i, \varphi_i, (a_0, \dots b_k) \right)$ is in  $G\text{-}\PreCob^V_k(n)$ for all $1 \leq i \leq m$ and the submanifolds $M_1,\dots M_m \subset V\times (a_0,b_k)$ are pairwise disjoint (compare with \cite[Section~7.1]{CalaqueScheimbauer}). 
The simplicial structure and topology can be defined as for $G\text{-}\PreCob^V(n)$. There is an induced symmetric monoidal structure on $G\text{-}\PreCob(n)$ by taking the homotopy colimit over subspaces of an infinite-dimensional vector space $\R^\infty$.   
   
\subsection{Homotopical equivariant topological field theories}
		We can now define a homotopical version of the equivariant topological field theories in \cite{turaevhqft}: 
\begin{definition}\label{defhomotopicalequivtft}
For any group $G$ an \emph{$n$-dimensional homotopical equivariant topological field theory with values in a symmetric monoidal $(\infty,1)$-category $\cat{S}$} is a symmetric monoidal $\iC$-functor 
\begin{align}
Z: G\text{-}\Cob(n) \to \cat{S}
\end{align}
We will refer to $\cat{S}$ as the \emph{target (category)} of the theory.
\end{definition}

\begin{remark}
In \cite{turaevhqft} equivariant topological field theories are not only symmetric monoidal functors on the category of $G$-cobordisms, but additionally satisfy the axiom of homotopy invariance, i.e.\ bordisms equipped with different maps to $BG$ gives rise to the same linear map if the maps to $BG$ are homotopic relative boundary. 
In the homotopical framework a higher version of homotopy invariance is implemented automatically because homotopies of classifying maps give rise to higher invertible morphisms in the $G$-cobordism category. Hence, passing to homotopy categories, a homotopical equivariant topological field theory induces an equivariant topological field theory. 
\end{remark}

\begin{remark}\label{remprecob}
	When considering homotopical equivariant topological field theories in their entirety we will be interested in the derived mapping space from $G\text{-}\Cob(n)$ to $\cat{S}$. Since derived mapping spaces are homotopy invariant, we can replace $G\text{-}\Cob(n)$ by any convenient equivalent model.
	In particular, since by definition $G\text{-}\Cob(n)$ is weakly equivalent to $G\text{-}\PreCob(n)$, we can describe maps from $G\text{-}\Cob(n)$ to a complete Segal space by elements in the derived mapping space from $G\text{-}\PreCob(n)$ to this complete Segal space, see \cite[2.2.8]{lurietft}. Hence, we can and will work with the uncompleted version $G\text{-}\PreCob(n)$ of the $(\infty,1)$-category of $G$-cobordisms. 
\end{remark}

			\begin{example}[Example for a target category: Cospans in $E_\infty$-algebras]\label{excospans}
For the main constructions of the present article we will be interested in the following target: 
Given an $(\infty,1)$-category $\cat{C}$ with finite colimits, we obtain by dualizing the results of \cite{haugseng} an $(\infty,1)$-category $\Cospan (\cat{C})$ of cospans in $\cat{C}$. When presented as a Segal space this $(\infty,1)$-category can be explicitly described in degree $k\ge 0$ as follows:
First denote by $\Sigma^k$ the partially ordered set of pairs $(i,j)$ such that $0\le i\le j\le k$, where $(i,j) \le (i',j')$ if $i\le i'$ and $j'\le j$ \cite[Definition~5.1]{haugseng}. In fact, $\Sigma^\bullet$ yields a cosimplicial object in categories. 
Now $\Cospan (\cat{C})_k$ is given by the maximal $\infty$-groupoid of the space of co-Cartesian functors from the category $\Sigma^k$ to $\cat{C}$. For a definition of co-Cartesian functors we refer to (the dual of) \cite[Definition~5.5]{haugseng} and the illustrations below. Informally speaking, the squares formed by a co-Cartesian functor $\Sigma^k\to \cat{C}$ are all homotopy pushouts.

By \cite[Corollary~8.5]{haugseng}, $\Cospan (\cat{C})$ is a complete Segal space, the simplical structure is induced by the cosimplical structure of $\Sigma^\bullet$. Using the coproduct of $\cat{C}$ we can turn $\Cospan (\cat{C})$ into a symmetric monoidal $(\infty,1)$-category. This is dual to \cite[Corollary~12.1]{haugseng}. Below we will use as target category $\Cospan (\cat{C})$ for $\cat{C}$ equal to the $(\infty,1)$-category of $E_\infty$-algebras in chain complexes. 

Concretely, an $\ell$-simplex of $\Cospan (\EAlg)_k$ is a homotopy coherent diagram $\mathfrak{C}(\Sigma^k \times [\ell])\longrightarrow \EAlg$ 
\begin{itemize}
\item which is co-Cartesian in $\Sigma^k$ 
\item and sends all morphisms in $[\ell]=0\to 1 \to \dots \to \ell$ to weak equivalences.
\end{itemize}
It is instructive to illustrate the definition in low degrees:
\begin{itemize}
\item 
A 0-simplex of $\Cospan (\EAlg)_0$ is an $E_\infty$-algebra $E_{0,0}$.

\item 
A 0-simplex of $\Cospan (\EAlg)_1$ is a diagram in $\EAlg$:
\begin{equation}
\begin{tikzcd}[column sep=small]
 &  E_{0,1} &  \\
 E_{0,0} \ar{ru} &  & E_{1,1} \ar{lu}
\end{tikzcd} \ .
\end{equation}

\item 
A 1-simplex of $\Cospan (\EAlg)_1$ is a homotopy coherent diagram in $\EAlg$:
\begin{equation}
\begin{tikzcd}[row sep=scriptsize]
 &  &  E_{0,1}' &  \\
 &  E_{0,1} \ar{ru}{\simeq} &   & \\ 
  & E_{0,0}' \ar{ruu} &  & E_{1,1}' \ar{luu}\\
 E_{0,0} \ar{ruu} \ar[swap]{ru}{\simeq} &  & E_{1,1} \ar[crossing over]{luu}  \ar{ru}{\simeq}& 
\end{tikzcd}     \ .
\end{equation}

\item 
A 0-simplex of $\Cospan (\EAlg)_2$ is a co-Cartesian homotopy coherent diagram in $\EAlg$:
\begin{equation}
\begin{tikzcd}[column sep=small]
 &   & E_{0,2}  &  &  \\
 &  E_{0,1} \ar{ru} & &  \ar{lu} E_{1,2} &  \\
 E_{0,0} \ar{ru} &  & E_{1,1} \ar{lu} \ar{ru} &  & E_{2,2} \ar{lu}
\end{tikzcd} \ .
\end{equation}
Being co-Cartesian means that the square
\begin{equation}
\begin{tikzcd}[column sep=small]
 &   & E_{0,2}  &  &  \\
 &  E_{0,1} \ar{ru} & &  \ar{lu} E_{1,2} &  \\
  &  & E_{1,1} \ar{lu} \ar{ru} &  
\end{tikzcd}
\end{equation}
is a homotopy pushout. 
\end{itemize}
We should comment on the relation to other common targets used in the literature, namely (higher) Morita categories of $E_n$-algebras: Their objects are $E_n$-algebras, their 1-morphisms are bimodules in $E_{n-1}$-algebras etc. 	
When we observe that for $n=\infty$ an $E_{n-1}$-algebra is again an $E_\infty$-algebra and that a cospan 
\begin{equation}
\begin{tikzcd}[column sep=small]
&  E_{0,1} &  \\
E_{0,0} \ar{ru} &  & E_{1,1} \ar{lu}
\end{tikzcd}
\end{equation}
of $E_\infty$-algebras
makes $E_{0,1}$ an $(E_{0,0} ,   E_{1,1})$-bimodule, we see that the idea behind the cospan category is not substantially different. In particular, we can see any (equivariant) topological field theory with values in cospans of $E_\infty$-algebras as a theory taking values in the Morita category of $E_\infty$-algebras. 
\end{example}

\subsection{Main Theorem\label{mainthm}}
Finally, we state our main result that equivariant higher Hochschild homology gives rise to a homotopical equivariant topological field theory.

\begin{theorem}\label{thmmain}
For a group $G$, let $A$ be a $G$-equivariant $E_\infty$-algebra in chain complexes over a field. 
Then $G$-equivariant higher Hochschild chains
with coefficients $A$
 naturally extend, for any $n\ge 1$, to an $n$-dimensional homotopical equivariant topological field theory
\begin{align} Z_A = \int_? A : G\text{-}\Cob(n) \to \Cospan (\EAlg)  \label{thehhqft}
\end{align} with values in cospans of $E_\infty$-algebras. 
\end{theorem}

The proof of this Theorem will occupy the rest of this section. 
In order to define \eqref{thehhqft} we use Remark~\ref{remprecob} and give $Z_A$ as a map $G\text{-}\PreCob(n) \to \Cospan (\EAlg)$, i.e.\ we need to exhibit for each fixed $k\ge 1$ a simplicial map
\begin{align}
G\text{-}\PreCob_k(n)  \to \Cospan (\EAlg)_k\ .\label{eqndeftft}
\end{align} 
It sufices to give a simplicial map 
\begin{align}
G\text{-}\PreCob_k^V(n)  \to \Cospan (\EAlg)_k \label{eqndeftftV}
\end{align} 
for each finite-dimensional real vector space $V$, which is natural in $V$.
Before giving \eqref{eqndeftftV} for general $\ell$-simplices we will discuss its definition on 0-simplices (using the notation in \eqref{eqn5tupel}): 
For a 5-tuple
\begin{align}
(M,f :  M \to B\SO(n),\psi:  M \to BG, \alpha :  f^* \xi \cong TM  ,\underline{I}) 
\end{align}
we built a functor $\Sigma^k \to \Top / BG$ sending $(i,j)$ with $1\le i\le j\le k$ to $M_{ij}$, which we define as the preimage of the closed interval in $(a_0,b_k)$ going from $a_i$ to $b_j$ under $\pi :   M \to V\times (a_0,b_k) \to (a_0,b_k)$,  equipped with the appropriate restriction $\psi_{ij} : M_{ij} \to BG$ of $\psi$. 
Hence, we assign to $(i,j)$ a certain submanifold of $M$ and restrict the map to $BG$ to it. The obvious submanifold inclusions and restrictions of maps to $BG$ yield the definition of the functor $\Sigma^k \to \Top / BG$ on morphisms. This functor is co-Cartesian, and so is the functor $\Sigma^k \to \sSet / BG$ obtained by taking singular simplices on the topological spaces involved. 
Taking $G$-equivariant higher Hochschild homology yields a functor $\Sigma^k \to \EAlg$. For this to be a 0-simplex in $\Cospan (\EAlg)_k$, we need this functor to be co-Cartesian (according to Example~\ref{excospans}). This is indeed the case 
due to the excision property for equivariant higher Hochschild homology formulated below.
Hence, we have constructed the map \eqref{eqndeftftV} on zero simplices.
After formulating and proving the excision property below (Proposition~\ref{satzexcision}) we proceed with the definition of \eqref{eqndeftftV} on higher simplices.

\subparagraph{Excision for equivariant higher Hochschild homology.}
The main computational tool for ordinary (i.e.\ non-equivariant) Hochschild homology is excision \cite[Theorem~1~(3)]{gtz}, i.e.\ the fact that derived Hochschild chains preserve homotopy pushouts. The following is an equivariant generalization taking into account not only the gluing of spaces, but also of maps to $BG$. The statement will be already adapted to the case of interest, namely the gluing of bordisms.

	\begin{proposition}[Excision]\label{satzexcision}
			Let $M: \Sigma_0 \to \Sigma_2$ be an oriented $n$-dimensional bordism (embedded in some Euclidean space) that is cut at a hypersurface $\Sigma_1$ giving  bordisms $M_1 : \Sigma_0 \to \Sigma_1$ and $M_2 : \Sigma_1 \to \Sigma_2$. Then for any map $\psi : M \to BG$ 
			the square
			\begin{equation}
			\begin{tikzcd}
		\int_ {	\psi|_{\Sigma_1}  }    A \ar{rr}{} \ar[swap]{dd}{} & &  \int_{\psi|_{M_1} } A       \ar{dd}{} \\
			& & \\
		\int_	{\psi|_{M_2} } A  \ar{rr}{} & &  \int_{\psi} A     
			\end{tikzcd}\label{eqnexcision}
			\end{equation} is a homotopy pushout of $E_\infty$-algebras. 
			\label{hochschildexcision}
		\end{proposition}
		\begin{proof} Note that
			\begin{equation}
			\begin{tikzcd}
			\psi|_{\Sigma_1} : \Sigma_1 \to BG \ar{rr}{} \ar[swap]{dd}{} & & \psi|_{M_1}  : M_1   \to BG \ar{dd}{} \\
			& & \\
			\psi_{M_2}  : M_2   \to BG  \ar{rr}{} & & \psi :M \to BG
			\end{tikzcd} \label{eqnpushoutproofex1}
			\end{equation} 
			is a pushout in $\sSet / BG$. 
			The functors $?^* EG : \sSet / BG\to \sSet$, $? \times X : \sSet \to \sSet$
			 and $? \otimes QA : \sSet \to \EAlg$ preserve pushouts (for $?\times X$ this is clear, for $?^* EG$ it follows from \cite{barmeierschweigert}, see also \cite[Example~2.10]{extofk}), and so do the coends.
			 Therefore, we can read off from \eqref{defeqnZA} that $\int_? A$ sends the pushout \eqref{eqnpushoutproofex1} to a pushout. The span underlying \eqref{eqnpushoutproofex1} has cofibrant vertices and cofibrant legs, so by Lemma~\ref{lemmacofbetcofobj} the same is true for the span underlying \eqref{eqnexcision}. Now we can invoke \cite[Proposition~A.2.4.4]{htt} to deduce that \eqref{eqnexcision} is a homotopy pushout.\end{proof}

\subparagraph{Grothendieck construction.}
In order to define \eqref{eqndeftftV} on $\ell$-simplices we need to assign to an $\ell$-simplex of $G\text{-}\PreCob_k^V(n)$ a homotopy coherent diagram from  $[\ell]=0\to 1 \to \dots \to \ell$ 
to the maximal subgroupoid of co-Cartesian diagrams $\Sigma^k \to \EAlg$, i.e.\ an $\ell$-simplex in $\Cospan (\EAlg)_k$. 
To this end we use the $(\infty,1)$-version of the \emph{Grothendieck construction}, also called \emph{(un)straightening} in \cite[Section~2.2.1 \& 3.2]{htt}. 

More precisely, we will use the version of the (un)straightening constructed in \cite[Section~3.2.4]{htt} that we will now briefly recall: For a small category $\cat{C}$
there is the \emph{straightening functor}
\begin{align} \St_\cat{C} : \sSet / N\cat{C} \to \sSet^\cat{C} \label{straighteqn}  \end{align} from simplicial sets over the nerve of $\cat{C}$ to $\cat{C}$-shaped diagrams in simplicial sets (hence, this construction produces in fact strict diagrams, not only homotopy coherent ones). It is left adjoint to the \emph{unstraightening functor} 
\begin{align}
 \Un_\cat{C} : \sSet^\cat{C}   \to \sSet / N\cat{C}. 
\end{align}
By \cite[Remark~3.2.5.7]{htt} for any functor $f: \cat{C} \to \cat{D}$ between categories the square
	\begin{equation}
\begin{tikzcd}
 \sSet^\cat{D} \ar{rr}{ \Un_\cat{D}} \ar[swap]{dd}{f^*} & & \sSet / N\cat{D} \ar{dd}{Nf^*} \ar[Rightarrow, shorten >= 10, shorten <= 10]{ddll}{\cong} \\
& & \\
\sSet^\cat{C} \ar{rr}{ \Un_\cat{C}} & & \sSet / N\cat{C}
\end{tikzcd} 
\end{equation} commutes up to canonical natural isomorphism, i.e.\
the unstraightening is compatible with pullbacks.
Here $f^* : \sSet^\cat{D} \to \sSet^\cat{C}$ is given by precomposition with $f$ and $Nf^* : \sSet / N\cat{D} \to \sSet / N\cat{C}$
is given by pullback along $Nf : N \cat{C}\to N\cat{D}$. Since for every category $\cat{C}$ the adjunction $\St_\cat{C} \dashv  \Un_\cat{C}$ 
is a Quillen equivalence when $\sSet / N\cat{C}$ is equipped with the covariant model structure 
and $\sSet^\cat{C}$ with projective model structure \cite[Proposition~3.2.5.18]{htt}, 
we can deduce that the square
	\begin{equation}
\begin{tikzcd}
\sSet / N\cat{D} \ar{rr}{Nf^* } \ar[swap]{dd}{\St_\cat{C}} & & \sSet / N\cat{C}  \ar{dd}{\St_\cat{D}} \ar[Rightarrow, shorten >= 10, shorten <= 10]{ddll}{\simeq} \\
& & \\
\sSet^\cat{D} \ar{rr}{f^*} & & \sSet^\cat{C}
\end{tikzcd} 
\end{equation} 
commutes up to canonical natural weak equivalence, i.e.\ straightening is compatible with pullbacks (in a slightly weaker sense than the unstraightening).

After these preparations we can start to define \eqref{eqndeftftV} on $\ell$-simplices: 
Let 
\begin{align}
		(M,f :  M \to B\SO(n),\psi:  M \to BG, \alpha  ,(a_0(s),\dots , b_k(s))_{s\in |\Delta_\ell| }) \ , \label{anellsimplex2}
\end{align} be an $\ell$-simplex of $G\text{-}\PreCob_k^V(n)$, where we use the notation from \eqref{anellsimplex}.
For fixed $(i,j) \in \Sigma^k$ we have a morphism $[1] \to [k]$ sending $0$ to $i$ and $1$ to $j$. 
This morphism gives rise to a map sending \eqref{anellsimplex2} to an $\ell$-simplex 
\begin{align}
		(M_{ij} ,f_{ij}  :  M \to B\SO(n),\psi_{ij}:  M_{ij} \to BG, \alpha _{ij} ,(a_i(s), b_i(s), a_j(s), b_j(s))_{s\in |\Delta_\ell| }) \ , 
\end{align}
in $G\text{-}\PreCob_1^V(n)$. 

Denote by $\pi_\ell^{ij} : M_{ij} \to |\Delta_\ell|$ the corresponding fibre bundle. The image $\Sing \pi_\ell^{ij} : \Sing M_{ij} \to \Sing |\Delta_\ell|$ under the functor $\Sing : \Top \to \sSet$ taking singular simplices is a fibration 
and so is the pullback
	\begin{equation}
			\begin{tikzcd}
		\widetilde M_{ij}  \ar{rr}{} \ar[swap]{dd}{\widetilde \pi_\ell^{ij}} & & \Sing M_{ij} \ar{dd}{\Sing \pi_\ell^{ij}} \\
			& & \\
			\Delta_\ell  \ar{rr}{} & & \Sing |\Delta_\ell|
			\end{tikzcd} \label{pullbackalonguniteqn}
			\end{equation}
			of $\Sing \pi_\ell^{ij}$ along the unit $\Delta_\ell  \to \Sing |\Delta_\ell|$ of the adjunction between geometric realization $|?| : \sSet \to \Top$ and singular simplices $\Sing : \Top \to \sSet$. Note that this pullback is also a homotopy pullback. 
		Using the straightening \eqref{straighteqn}, $\widetilde \pi_\ell^{ij}$ corresponds to a functor $F_\ell^{ij} : [\ell]  \to \sSet$, where $[\ell]$ is the category $0 \to 1 \to \cdots \to \ell$. 
		Next we consider the commuting square
			\begin{equation}
		\begin{tikzcd}
		\widetilde M_{ij}  \ar{rr}{} \ar[swap]{dd}{\widetilde \pi_\ell^{ij}} & & BG \ar{dd}{} \\
		& & \\
		\Delta_\ell  \ar{rr}{t} & & \star
		\end{tikzcd} \label{eqnGrothendieck2}
		\end{equation} which can be read as a morphism of fibrations over $\Delta_\ell$ going from $\widetilde \pi_\ell^{ij}$ to the pullback $t^* BG$ of the fibration $BG \to \star$ along the unique map $\Delta_\ell \to \star$. 
		Therefore, the functoriality of the straightening gives us a map $F_\ell^{ij} \to \St_{\Delta_\ell} t^*BG$ of diagrams of shape $[\ell]$ with values in simplicial sets. 
		Because of the straightening's compatibility with the pullback we obtain a canonical weak equivalence
		$\St_{\Delta_\ell} t^* BG \to t ^* \St_\star BG$, and by \cite[Remark~3.2.5.5]{htt} $t ^* \St_\star BG$ is just the constant diagram with value $BG$. 
		In summary, we obtain a lift $\widehat{F}_\ell^{ij}: [\ell] \to \sSet /BG$ of $F_\ell^{ij}$ to a diagram taking values in $\sSet /BG$. By taking equivariant higher Hochschild homology we obtain a diagram $\int_{ \widehat{F}_\ell^{ij}} A : [\ell] \to \EAlg$. It is functorial in $(i,j) \in \Sigma^k$.
		Therefore we obtain a diagram of shape $[\ell]$ in the maximal subgroupoid of co-Cartesian functors from $\Sigma^k$ to $E_\infty$-algebras. Let us explain in more detail:
		
		\begin{itemize}
		\item
		
The fact that this construction really takes values in the maximal $\infty$-subgroupoid can be seen as follows: 
Since, by \eqref{pullbackalonguniteqn}, $\widetilde M_{ij}   \to 	\Delta_\ell$ is the pullback of $\Sing M_{ij} \to \Sing |\Delta_\ell|$, its straightening factors -- up to weak equivalence -- through $\mathfrak{C}[\Sing|\Delta_\ell|]$, see \cite[Section~1.1.5]{htt} for the definition. Now the claim follows from the fact that $\Sing|\Delta_\ell|$ is a Kan complex, i.e.\ an $\infty$-groupoid.


\item

		The fact that we produce indeed co-Cartesian functors follows again from excision (Proposition~\ref{satzexcision}). 
			\end{itemize}
		This concludes the definition of \eqref{eqndeftftV} on simplices of arbitrary degree.

		We should emphasize a technical point: When verifying that the maps \eqref{eqndeftftV} are simplicial (in $\ell$), we will notice that the compatibility of the straightening with pullbacks only allows our map to be simplicial \emph{up to coherent homotopy}. However, the map becomes an honest simplicial map once we replace $G\text{-}\PreCob^V(n)$ by an equivalent model, which we are allowed to do since we are only interested in the derived mapping space from $G\text{-}\Cob(n)$ to $\Cospan (\EAlg)$, see Remark~\ref{remprecob}.
		The needed replacement is performed by the following construction: We denote by $U : \sSet \to \Set^{\mathbb{N}_0}$ the forgetful functor from simplicial sets to $\mathbb{N}_0$-graded sets. By left Kan extension, it has a left adjoint $F: \Set^{\mathbb{N}_0} \to \sSet$ given by
		\begin{align} FS = \coprod_{k\in\mathbb{N}_0} \Delta_k \times S_k \end{align} for $S \in \Set^{\mathbb{N}_0}$. The monad corresponding to this adjunction gives us a simplicial bar construction for each $X \in \sSet$, i.e.\ an augmented simplicial object in simplicial sets
		\begin{align}
\underbrace{	\dots \substack{ \longrightarrow \\ \longleftarrow \\ \longrightarrow \\ \longleftarrow \\ \longrightarrow }  F^2 X  \substack{ \longrightarrow \\ \longleftarrow \\ \longrightarrow }FX}_{\Bar_\bullet X} \to X .
		\end{align} Here we suppress the forgetful functor in the notation. The geometric realization 
		\begin{align} \Bar X := |\!\Bar_\bullet X| = \int^{n} \Delta_n \times F^{n+1} X\label{eqnthecoend} \end{align}
		of the simplicial bar construction yields the bar construction of $X$. Since the augmentation $\Bar_\bullet X \to X$ admits extra degeneracies, the induced map $\Bar X \to X$ is a weak equivalence. 
		
		Applying the bar construction just outlined level-wise to the simplicial space describing the $G$-bordism category solves the issue that the maps \eqref{eqndeftftV} are only simplicial \emph{up to coherent homotopy} in $\ell$. The underlying problem is that the straightening $\phi_\bullet := \St_{[\bullet]}$ between the simplicial objects $\sSet / \Delta_\bullet$ and $N \sSet = \sSet^{[\bullet]}$ is only simplicial up to homotopy, which means that for any morphism $f: [m]\to[n]$ in $\Delta$ the square
			\begin{equation}
		\begin{tikzcd}
		\sSet / \Delta_n \ar{rr}{\Delta_f^* } \ar[swap]{dd}{\phi_n := \St_{[n]}} & & \sSet /\Delta_m  \ar{dd}{\phi_m := \St_{[m]}} \ar[Rightarrow, shorten >= 10, shorten <= 10]{ddll}{\simeq} \\
		& & \\
	N_n \sSet =	\sSet^{[n]} \ar{rr}{f^*} & &N_m \sSet =  \sSet^{[m]}
		\end{tikzcd} 
		\end{equation} commutes up to canonical natural weak equivalence.
		 
(Note that for the construction of the equivariant field theories it is essential that these maps are weak equivalences, since our construction has to land in the maximal subgroupoid of co-Cartesian diagrams.)

In particular, we get for $\sigma \in 	\sSet / \Delta_n$ a natural weak equivalence $\phi_m \Delta_f^* \sigma \to f^* \phi_n \sigma$ in $\sSet^{[m]}$. This is equivalently a functor
		\begin{align}
		h(f,\sigma) : [m] \times [1] \to \sSet,
		\end{align} whose restriction to $0$ and $1$ is $\phi_m \Delta_f^* \sigma$ and $f^* \phi_n \sigma$, respectively. 
		Using that the nerve is fully faithful we obtain
		\begin{align}
		\sSet (\Delta_p \times \Delta_q, N\sSet ) \cong  ( N ([p] \times [q])  ,N\sSet ) \cong \sSet^{ [p]\times [q] } \ .
		\end{align} 
		We can see $h(f,\sigma)$ also as a simplicial map making the diagram
\begin{equation}\label{eqnhtpassmap}
\begin{tikzcd}[column sep=large, row sep=large]
\Delta_m \ar[hookrightarrow]{d}{\times 0}\ar{rrd}{\phi_m \Delta_f^* \sigma} & &  \\
 \Delta_m \times \Delta_1 \ar[hookleftarrow]{d}{\times 1} \ar{rr}{h(f,\sigma)}& & N\sSet \\
\Delta_m  \ar[swap]{rru}{f^* \phi_n \sigma} & & 
\end{tikzcd}
\end{equation} commute.
We denote this map by the same symbol.

We will now show that $\phi_\bullet$ induces an honest simplicial map $\widetilde \phi_\bullet : \Bar (\sSet / \Delta_\bullet) \to N\sSet$. By definition such a map consists of compatible maps
\begin{align}
\Delta_p\times F^{p+1}X \longrightarrow N\sSet \end{align} of simplicial sets which, by adjunction, are in bijection to maps \begin{align} F^{p+1}X \longrightarrow N\sSet^{\Delta_p} .
\end{align}
Using that $F$ is a free functor we can describe these maps  equivalently as maps of sets
\begin{align}
(F^p \sSet / \Delta_\bullet)_m \to ((N\sSet)^{\Delta_p})_m= \sSet ( \Delta_p \times \Delta_m ,N\sSet   ) \quad \text{for all}\quad  m \in \N_0
\end{align} 
compatible as prescribed by the coend \eqref{eqnthecoend}.
		For the definition in the case $p=0$ we can obviously use $\phi_*$. For $p=1$, i.e.\ for the definition of
		\begin{align}
		(F \sSet / \Delta_\bullet)_m = \coprod_{n \in \mathbb{N}_0} \Delta(m,n) \times \sSet/\Delta_n \to  \sSet ( \Delta_m \times \Delta_1 , N\sSet  ),
		\end{align} we can use the homotopies \eqref{eqnhtpassmap}. 

For $p=2$ we have to construct maps 
\begin{align}\label{Eq: p=2}
(F^2 \sSet / \Delta_\bullet)_m = \coprod_{n ,k \in \mathbb{N}_0} \Delta(m,k)\times \Delta(k,n) \times \sSet/\Delta_n \to  \sSet ( \Delta_m \times \Delta_2 , N\sSet  ) \ .
\end{align}
Homotopy coherence implies that for composable morphisms $f\in \Delta(m,k)$ and $g\in \Delta(k,n)$ in the simplex category and $\sigma : M \longrightarrow \Delta_n \in \sSet/\Delta_n $ the diagram 
\begin{equation}
\begin{tikzcd}
 & f^*\phi_k \Delta_g^* \sigma \ar{rd}{h(g,\sigma)} \ar[Rightarrow, shorten >= 5, shorten <= 5]{d} & \\
\phi_m \Delta_f^* \Delta_g^* \sigma \ar[swap]{rr}{h(g\circ f,\sigma)} \ar{ru}{h(f,\sigma)} & \ & f^* g^* \phi_n \sigma 
\end{tikzcd}
\end{equation}
commutes up to homotopy $\Delta_m\times \Delta_2 \to N \sSet$. These homotopies are used to define the maps \eqref{Eq: p=2}.

In general we have to construct maps of the following form 
\begin{align}
(F^p \sSet / \Delta_\bullet)_m = \coprod_{n_1, \dots , n_p \in \mathbb{N}_0} \Delta(m,n_p)\times \cdots \times \Delta(n_2,n_1) \times \sSet/\Delta_{n_1} \to  \sSet ( \Delta_m \times \Delta_p , N\sSet  ) \ .
\end{align}   
As above the compatibility of the straightening with pullback up to coherent homotopy yields precisely the desired maps.
These maps are compatible in the sense that they descend to the coend \eqref{eqnthecoend}.

\begin{proof}[Proof (of Theorem~\ref{thmmain})]
The prescribed maps \eqref{eqndeftftV} are simplicial (in $\ell$) after resolving $G\text{-}\PreCob^V(n)$ as just expained. Furthermore, \eqref{eqndeftftV} is obviously natural in $V$ and hence yields the map \eqref{eqndeftft}.

		Next we show that \eqref{eqndeftftV} is simplicial in $k$:	
Let $g  :  [m]\longrightarrow [k]$ be a map in $\Delta$ and 
\begin{align}
(M,f :  M \to B\SO(n),\psi:  M \to BG, \alpha :  f^* \xi \cong TM  ,\underline{I}) 
\end{align} 
a zero simplex of $G\text{-}\PreCob^V(n)_{k}$.
Applying the map $\Cospan(\EAlg)_{k,0}\longrightarrow \Cospan(\EAlg)_{m,0}$ induced by $g$ to the image of $(M,f :  M \to B\SO(n),\psi:  M \to BG, \alpha :  f^* \xi \cong TM  ,\underline{I})$ results in the homotopy coherent diagram
\begin{align}
\Sigma^{m} &\longrightarrow \EAlg \\
(i,j) & \longmapsto Z_A(M_{g(i),g(j)}) \ .
\end{align}
The formula for higher $\ell$-simplices is similar.
Comparing with \eqref{Eq: Simplicial structure for GCob} completes the proof that $Z_A$ defines a map of simplicial spaces, i.e.\ an $\iC$-functor.
\color{black}

To endow $Z_A$ with the structure of a symmetric monoidal functor 
we need to specify maps of Segal spaces 			
\begin{align}
G\text{-}\PreCob^V(n)[m] \to \Cospan (\EAlg)^m \ ,
\end{align}
natural in the finite dimensional vector space $V$, such that
$G\text{-}\PreCob^V(n)[1]  \to \Cospan (\EAlg) $ corresponds to the construction
given above.
Recall from \eqref{Eq: Element Gamma Space} that an element of $G\text{-}\PreCob^V(n)[m]$ is given by 
\begin{align}\label{Eq: Element Gamma Space}
\left( M_1,\dots ,M_m; f_1,\dots ,f_m; \psi_1,\dots ,\psi_m; \varphi_1,\dots ,\varphi_m; (a_0, \dots ,b_k) \right) \ .
\end{align}
Applying the construction described above to $( M_i; f_i; \psi_i; \varphi_i; (a_0, \dots b_k))$ for all $i$ produces $m$ different elements of $\Cospan (\EAlg)$ or equivalently an element of $\Cospan (\EAlg)^m$.

We need to show that 
\begin{equation}
\begin{tikzcd}
G\text{-}\PreCob^V(n)[m] \ar{r} \ar[swap]{d}{\gamma_*} & \Cospan (\EAlg)^m \ar{d}{\gamma_*} \\
G\text{-}\PreCob^V(n)[k] \ar{r}  & \Cospan (\EAlg)^k\label{monoidalsquareeqn}
\end{tikzcd}
\end{equation}
commutes up to coherent homotopies for all morphisms $\gamma: \langle m\rangle \to \langle k\rangle \in \Gamma$. 
Recall that the induced map $G\text{-}\PreCob^V(n)[m] \longrightarrow G\text{-}\PreCob^V(n)[k]$ is given by the disjoint union over the preimages of $\gamma$. Similarly, the induced map $\Cospan (\EAlg)^m \to \Cospan (\EAlg)^k$ is given by taking the tensor product of $E_\infty$-algebras over the preimages of $\gamma$.  The commutativity of \eqref{monoidalsquareeqn} is implied by the following fact:
For maps $\varphi : M \to BG$ and $\varphi' : M' \to BG$ from manifolds (actually from arbitrary spaces) to $BG$ we find
\begin{align}
\Big( \varphi^*EG \sqcup \varphi'^*EG \Big) \stackrel{\mathbb{L}}{\otimes}_G A \cong  \left(\varphi^*EG \stackrel{\mathbb{L}}{\otimes}_G A \right) \otimes \left( \varphi'^*EG \stackrel{\mathbb{L}}{\otimes}_G A \right) 
\end{align}
by a canonical isomorphism.
\end{proof}

\begin{remark}\label{bordismhypothesisrem}
	The existence of the equivariant field theory from Theorem~\ref{thmmain} can also be derived from the cobordism hypothesis as we will sketch now:
	According to the cobordism hypothesis~\cite[Theorem~2.4.26]{lurietft} the space of $n$-dimensional $G$-equivariant field theories with values in an $(\infty,n)$-category $\mathcal{C}$ is given by the space of homotopy fixed points of the $\SO(n)\times G$-action on the $\infty$-groupoid $X$ of fully dualizable objects in $\mathcal{C}$. 
	The group $G$ acts trivially, i.e. the action factors through $\SO(n)$. 
	
	The space of homotopy fixed point of interest is given by $\Hom_{\SO(n)\times G}(E\SO(n)\times EG, X)$, i.e.\ by the space of $\SO(n)\times G$-equivariant maps $E\SO(n)\times EG \longrightarrow X$. We can rewrite this space as 
	\begin{align}
	\Hom_{\SO(n)\times G}(E\SO(n)\times EG, X) \cong \Hom_{G}(EG, \Hom_{\SO(n)}(E\SO(n), X)) \ ,
	\end{align}      
	where we should see $\Hom_{\SO(n)}(E\SO(n), X) $ as the space of oriented topological field theories. Since the $G$-action on $\Hom_{\SO(n)}(E\SO(n), X) $ is trivial, a $G$-equivariant topological field theory is the same as a map
	\begin{align}
	BG \to  \Hom_{\SO(n)}(E\SO(n), X)  ,
	\end{align}
	i.e.\ an oriented topological field theory with homotopy coherent $G$-action.  
	
	Now choose as the target $(\infty,n)$-category the category of iterated cospans in $\EAlg$. By~\cite[Theorem 13.3]{haugseng} every object in this category is fully dualizable. 
	
	A $G$-action $BG\longrightarrow \EAlg$ on an $E_\infty$-algebra $A$ can be 
	considered as an action in the cospan category by sending $g :  A\longrightarrow A$ to the cospan 
	\begin{equation}
	\begin{tikzcd}
	& A & \\
	A\ar{ru}{g} &  & A \ar[swap]{lu}{\id} 
	\end{tikzcd} \ .
	\end{equation}
	If we now assume that every $E_\infty$-algebra can be equipped with `a trivial' $\SO(n)$-fixed point structure in the $(\infty,n)$-category of cospans (corresponding to ordinary higher Hochschild homology), 
	we see that every $G$-action 
	on a given $E_\infty$-algebra induces a $G$-equivariant topological field theory. In this article we have constructed this theory explicitly, or rather its $(\infty,1)$-version.  	
	\end{remark}

\subsection{Examples and application}
We conclude the paper  with a few examples and applications that involve some concrete computations of equivariant Hochschild homology.
		
		\subsubsection{Equivariant higher Hochschild homology of the circle\label{excircle}}
	As a preliminary consideration leading to the computation of equivariant Hochschild homology of the circle below, consider the field theory 	\begin{align} Z_A : G\text{-}\Cob(n) \to \Cospan (\EAlg)
		\end{align} associated to a $G$-equivariant $E_\infty$-algebra $A$ by Theorem~\ref{thmmain}. 
		Given two objects $\varphi_0, \varphi_1 : \Sigma \to BG$ and a homotopy $\varphi_0 \stackrel{h}{\simeq} \varphi_1$, we can evaluate $Z_A$ on $h: \Sigma \times [0,1] \to BG$. The resulting cospan is given by
		\begin{align}
		Z_A (h: \Sigma \times [0,1] \to BG) = \left(  \varphi_0^* EG \stackrel{\mathbb{L}}{\otimes}_G   A \to h^* EG \stackrel{\mathbb{L}}{\otimes}_G A \longleftarrow \varphi_1^* EG \stackrel{\mathbb{L}}{\otimes}_G A    \right).
		\end{align} Using \begin{itemize}\item the isomorphism $\varphi_0^* EG \cong \varphi_1^* EG$ induced by $h$
		\item and the $G$-equivariant equivalence $h^*EG \to \varphi_1^* EG$ sending a pair $( (x,t) , p  ) \in h^* EG$ with $(x,t) \in \Sigma \times [0,1]$ and $p \in EG_{h(x,t)}$ to the parallel transport of $p$ along the path in $BG$ from $h(x,t)$ to $h(x,1)$ parametrized by the homotopy $h$,\end{itemize}
		 we find the commutative diagram
\begin{equation}
\begin{tikzcd}[ row sep=small]
 & h^*EG \stackrel{\mathbb{L}}{\otimes}_G A  \ar{dd}{\simeq} & \\
 \varphi_0^* EG \stackrel{\mathbb{L}}{\otimes}_G A \ar{ru} \ar{rd}{h} & & \varphi_1^* EG \stackrel{\mathbb{L}}{\otimes}_G A \ar{lu}\ar{ld}{=} \\
  & \varphi_1^* EG \stackrel{\mathbb{L}}{\otimes}_G A &
\end{tikzcd}
\end{equation}		
and conclude
		\begin{align}
			Z_A (h: \Sigma \times [0,1] \to BG)\simeq  \left(  \varphi_0^* EG \stackrel{\mathbb{L}}{\otimes}_G   A \stackrel{h}{\to} \varphi_1^* EG \stackrel{\mathbb{L}}{\otimes}_G A  \stackrel{=}{\longleftarrow} \varphi_1^* EG \stackrel{\mathbb{L}}{\otimes}_G A    \right).\label{prepeqn}
			\end{align}

		We use this to compute the value of 
		\begin{align} Z_A : G\text{-}\Cob(1) \to \Cospan (\EAlg)
		\end{align} on the circle decorated with a map $\varphi : \mathbb{S}^1 \to BG$: First note that the $G$-bundle $\varphi^* EG \to \mathbb{S}^1$ described by $\varphi$ is fully characterized by its holonomy $g \in G$ around the circle. 
For simplicity we will assume that $A$ is a commutative algebra.
We will use the decomposition 
	\begin{center}
		\begin{overpic}[scale=0.35
			,tics=10]
			{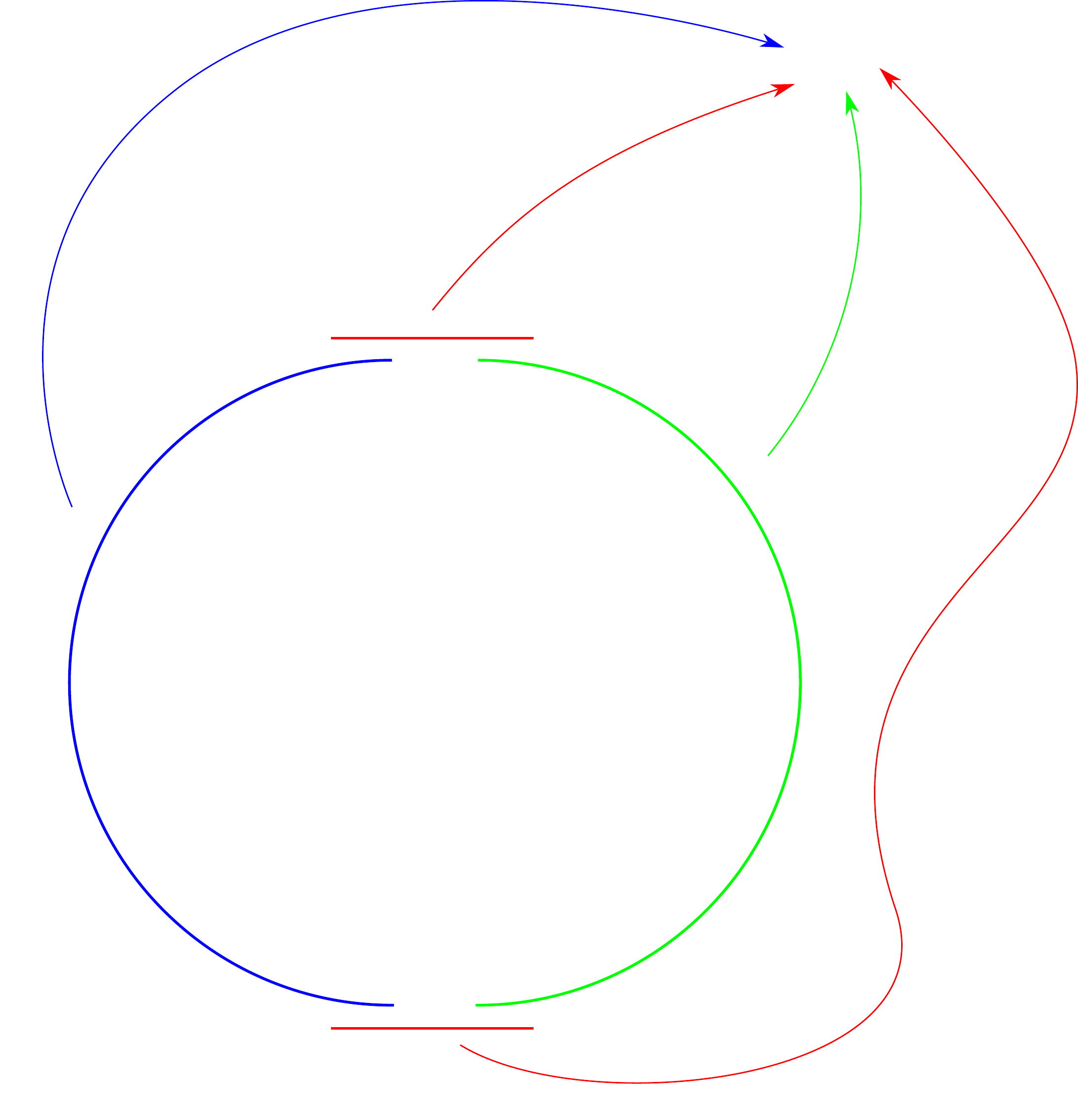}
			\put(33,3){\color{red}$A$}
			\put(33,72){\color{red}$A$}
			\put(0,35){\color{blue}$A$}
			\put(72,35){\color{green}$A$}
			\put(71,93){$BG$}
			\put(85,32){\color{red}$*$}
			\put(50,86){\color{red}$g$}
			\put(74,63){\color{green}$*$}
			\put(10,82){\color{blue}$*$}
		\end{overpic}
	\end{center}
of $\mathbb{S}^1$. Here $*$ denotes the constant map into $BG$.
According to this decomposition,  we get the element 
\begin{equation}\label{diagramcocart}
\begin{tikzcd}[row sep=scriptsize]
 &   &   & \hspace{-1cm}Z_A(\varphi : \mathbb{S}^1 \to BG)\hspace{-1cm} & & & \\
 &   & A \ar{ru}{}&  & \ar{lu}{} A & & \\
 & A \ar{ru}{} & & A\otimes A \ar[swap]{ru}{\mu } \ar{lu}{} & & \ar{lu}{} A & \\
k \ar{ru}{1} & & A \otimes A \ar[swap]{ru}{g\otimes \id} \ar{lu}{\mu} & & A\otimes A \ar[swap]{ru}{\mu } \ar{lu}{\id} & & \ar{lu}{1} k 
\end{tikzcd}
\end{equation}of $\Cospan (\EAlg)_4$,
where $\mu$ is the multiplication map.
Here we have used \eqref{prepeqn} and the fact
that for a half-circle decorated with the trivial $G$-bundle, the inclusion of the endpoints induces the multiplication $\mu : A \otimes A \to A$ of $A$. This is true for non-equivariant higher Hochschild homology (as can be extracted from \cite[Example~4 (1)]{gtz}) and hence by Lemma~\ref{lemmatrivialbundle} gives us the statement in the equivariant case as long as the bundle used for decoration is trivial.

Using that diagram~\eqref{diagramcocart} is co-Cartesian, we find the homotopy pushout square
		\begin{equation}
		\begin{tikzcd}
		A \otimes A  \ar{rr}{\mu \circ (g\otimes  \id_A)} \ar[swap]{dd}{\mu} & & A \ar{dd}{}  \\
		& & \\
		A  \ar{rr}{} & & Z_A(\varphi : \mathbb{S}^1 \to BG) 
		\end{tikzcd}\label{eqnexcision1}  .
		\end{equation}
 This homotopy pushout is the derived tensor product of the $A\otimes A$-modules $A$ with usual $A\otimes A$-action and $A$ with an $A\otimes A$-action twisted by $g$ as prescribed by the diagram (we call this $A\otimes A$-module $A_g$), i.e.\
		\begin{align}
		Z_A(\varphi : \mathbb{S}^1 \to BG)  = \int_{\varphi } A \simeq A \stackrel{\mathbb{L}}{\otimes}_{A\otimes A} A_g.\label{twsitedHHeqn}
 		\end{align}
 		Consequently, the homology of $	Z_A(\varphi : \mathbb{S}^1 \to BG)$ is the Hochschild homology of $A$ with coefficients in $A_g$. 
 		
 		In the special case $G=\mathbb{Z}$, the equivariant structure on $A$ is determined by an isomorphism $\phi : A \to A$. By \cite[Corollary~5]{ginot} this data can be used to model a locally constant factorization algebra on the circle; and twisted Hochschild homology as in \eqref{twsitedHHeqn} computes its factorization homology.

\subsubsection{Equivariant Dijkgraaf-Witten models and their orbifoldization\label{exdw}}
		Non-equivariant Dijk\-graaf-Witten theories are a type of topological field theory constructed from a finite group $G$ (the construction below will work for any discrete group). They were developed in \cite{freedquinn, morton1} based on \cite{dijkgraafwitten}. 
		The idea is to use the bundle groupoid of a manifold as a configuration space and to build the topological field theory from the (invariant) functions on that groupoid.

	Dijkgraaf-Witten theory is usually constructed with vector spaces as a target. We will try to transfer the ideas to obtain a homotopical topological field theory $\Cob(n) \to \Cospan (\EAlg)$ that can deservedly be called a Dijkgraaf-Witten theory. 
	
	Dijkgraaf-Witten theory is built from vector spaces of functions on the set of isomorphism classes of $G$-bundles over manifolds of different dimensions to the ground field $k$.
Locally, i.e.\ over a contractible manifold, we can see this vector space as the zeroth cohomology of the (normalized) $k$-cochains on the nerve of $\star \DS G$
	\begin{align}
	\dots \substack{ \longrightarrow \\ \longleftarrow \\ \longrightarrow \\ \longleftarrow \\ \longrightarrow }  G  \substack{ \longrightarrow \\ \longleftarrow \\ \longrightarrow }\star .
	\end{align} 
		In order to pass to a homotopical setting, it seems reasonable to replace functions on isomorphism classes by the entire complex of normalized cochains $N^*(\star \DS G; k)=N^*(BG;k)$ with coefficients in $k$.
		The cochains on a space form an $E_\infty$-algebra \cite{bergerfresse}, so we obtain from $N^*(\star \DS G; k)$ an 
		$(\infty,1)$-topological field theory
				\begin{align}
				Z_G := Z_{N^*(\star \DS G;k)} : \Cob(n) \to \Cospan (\EAlg)
				\end{align} that we call the \emph{$(\infty,1)$-Dijkgraaf-Witten theory for the group $G$.}
		Note that in the non-homotopical setting one requires that the ground field $k$ has characteristic zero to ensure that the representation theory of $G$ is semisimple. We do not need this for our considerations.

		The $E_\infty$-algebra that $Z_G$ assigns to a closed oriented $n$-dimensional manifold $M$ is the factorization homology of $M$ with coefficients in $N^*(\star \DS G;k)$. 
		Na\"ively, one might expect that this $E_\infty$-algebra 
		is given by the cochains on the groupoid of $G$-bundles over $M$. However, such a result is just available in the following special case: If $G$ is a finite nilpotent group, then by \cite[Proposition~5.3]{af} 
		\begin{align}
		Z_G(M) = \int_M N^*(\star \DS G;k) \simeq N^* (  BG^M;k   ) . 
		\end{align} Since the space $BG^M$ of maps $M \to BG$ is equivalent to the nerve of the groupoid $\PBun_G(M)$ of $G$-bundles over $M$, we find in that case
		\begin{align}
		Z_G(M) \simeq N^*(\PBun_G(M);k).\end{align}

		We will now interpret equivariant Dijkgraaf-Witten models in this fashion. These were constructed in \cite{maiernikolausschweigerteq} from a short exact sequence of groups. In \cite{schweigertwoikeofk, MuellerWoike1} the input was generalized to an arbitrary group morphism $\lambda : G \to H$ and non-trivial twists. The configuration spaces of the $H$-equivariant Dijkgraaf-Witten model that $\lambda$ gives rise to are homotopy fibers of the functor from $G$-bundles to $H$-bundles induced by $\lambda$. Over the point there is only one $H$-bundle and its homotopy fiber can easily be seen to be the action groupoid $H\DS G$, where $G$ acts on $H$ by $\lambda$ from the left. This space carries naturally a right action of the automorphism group of the $H$-bundle over the point, namely $H$, which turns $N^*(H\DS G;k)$ into an $H$-equivariant $E_\infty$-algebra. The associated $H$-equivariant field theory
		\begin{align}
		Z_\lambda := Z_{N^*(H\DS G;k)} : H\text{-}\Cob(n) \to \Cospan (\EAlg)
		\end{align} provided by Theorem~\ref{thmmain} will be referred to as the \emph{$H$-equivariant $(\infty,1)$-topological field theory induced by $\lambda$}. 
		
		So far, using the name \emph{Dijkgraaf-Witten} for the above theories is only loosely motivated, but we will show now that with these definitions we can find results analogous to those obtained for non-homotopical Dijkgraaf-Witten models. Our sample computation concerns orbifoldization \cite{maiernikolausschweigerteq, schweigertwoikeofk, extofk}, a construction which produces from a equivariant theory a non-equivariant theory by taking invariants. 
		
		For a $G$-equivariant $E_\infty$-algebra $A$ it appears reasonable to call the non-equivariant theory $Z_A /G := Z_{\holimsub{G} A}$ the \emph{orbifold theory} of the $G$-equivariant theory $Z_A$. 
		
		In order to compute the orbifold of $Z_\lambda$ we observe that taking cochains on a simplicial set is a left Quillen functor when considered as taking values in the opposite category of chain complexes. It furthermore preserves weak equivalences, from which we deduce that it takes homotopy colimits to homotopy limits, i.e.\
		\begin{align}
			\holimsub{H }N^*(H\DS G;k) \simeq N^* (\hocolimsub{H} H\DS G;k    ). 
			\end{align}
			Since the action of $H$ on the space $H\DS G$ is free,
			\begin{align}
			\hocolimsub{H} H\DS G \simeq \colimsub{H} H\DS G,
			\end{align} 
			and the right hand side can be easily seen to be given by $\star  \DS G$. 
			This proves
			\begin{align}
			\frac{Z_\lambda}{H} \cong Z_G.
			\end{align}
			In the non-homotopical setting this result was proven in \cite{schweigertwoikeofk, extofk} showing that the homotopical Dijkgraaf-Witten theories behave similarly to the non-homotopical ones.

\end{document}